\documentclass[10pt, letterpaper, twoside]{article}

\usepackage{amsmath, amsthm,amssymb,bbold,mathrsfs} 
\usepackage{amsfonts}
\usepackage{graphicx}
\usepackage[titletoc,title]{appendix}
\usepackage{enumitem}
\setlist[itemize]{itemsep=0pt,parsep=2pt,topsep=2pt}
\setlist[enumerate]{itemsep=0pt,parsep=2pt,topsep=2pt}
\usepackage[numbers]{natbib}
\usepackage[hidelinks]{hyperref}

\newtheorem{theorem}{Theorem}[section]
\newtheorem{prop}[theorem]{Proposition}%[section]
\newtheorem{lemma}[theorem]{Lemma}%[section]
\newtheorem{assumption}[theorem]{Assumption}%[section]
\theoremstyle{remark}
\theoremstyle{definition}
\newtheorem{example}{Example}[section]
\newtheorem{definition}[theorem]{Definition}%[section]
\theoremstyle{plain}
\newtheorem{myremark}{Remark}[section]
\numberwithin{equation}{section}

 \def\argmin{\mathop{\arg\min}}
 \def\A{\mathbb{A}}
 \def\X{\mathbb{X}}
 \def\E{\mathbb{E}}
 
 \def\R{\mathbb{R}}
 
 \def\C{\mathcal{C}}
 \def\F{\mathcal{F}}
 \def\P{\mathcal{P}}
 \def\B{\mathcal{B}}

  \def\0{\mathbf{0}}
 \def\1{\mathbf{1}}
 \def\ind{\mathbb{1}}
 
\def\iJ{\underline{J}}
\def\Pr{\mathbf{P}}
\def\proj{\mathop{\text{\rm proj}}}
\def\Z{\mathbb{Z}}
\def\Q0{{\mathbb{Q}_0}}
\def\N{\mathcal{N}}
\def\W{\mathcal{W}}
\def\wto{\mathop{\overset{\text{\tiny \it w}}{\rightarrow}}}

\def\myqed{\hfill \qed}

\begin{document} 
\markboth{On Minimum Pair Approach for Average-Cost MDP}{}

\title{On the Minimum Pair Approach for Average-Cost\\ Markov Decision Processes with Countable Discrete Action Spaces and Strictly Unbounded Costs\thanks{This research was funded by DeepMind and Alberta Innovates---Technology Futures.}
\thanks{This paper includes and extends the author's earlier results given in the arXiv eprint \cite[Section 4]{Yu19} (version 1).}}

\author{Huizhen Yu\thanks{RLAI Lab, Department of Computing Science, University of Alberta, Canada (\texttt{janey.hzyu@gmail.com})}}
\date{}

\maketitle

\begin{abstract}
We consider average-cost Markov decision processes (MDPs) with Borel state spaces, countable, discrete action spaces, and strictly unbounded one-stage costs. For the minimum pair approach, we introduce a new majorization condition on the state transition stochastic kernel, in place of the commonly required continuity conditions on the MDP model. We combine this majorization condition with Lusin's theorem to prove the existence of a stationary minimum pair, i.e., a stationary policy paired with an invariant probability measure induced on the state space, with the property that the pair attains the minimum long-run average cost over all policies and initial distributions. We also establish other optimality properties of a stationary minimum pair, and for the stationary policy in such a pair, under additional recurrence or regularity conditions, we prove its pathwise optimality and strong optimality. Our results can be applied to a class of countable action space MDPs in which the dynamics and one-stage costs are discontinuous with respect to the state variable.
\end{abstract}

strong and pathwise average-cost optimality; majorization condition

%\begin{AMS} 90C39, 90C40, 93E20 \end{AMS}

\clearpage
\tableofcontents

\clearpage
\section{Introduction}

We study discrete-time Markov decision processes (MDPs) under the long-run average cost criteria. 
Specifically, we consider MDPs with Borel state spaces, countable action spaces with the discrete topology, and one-stage costs that are nonnegative and strictly unbounded. We study optimality properties of these MDPs by using the minimum pair approach.  

A minimum pair for an MDP is a policy and an initial state distribution with the property that the pair attains the minimum (limit superior) expected average cost over all policies and initial state distributions. Of interest is the existence of a minimum pair with special structures, in particular, a stationary policy and an invariant probability measure induced by the policy on the state space. We shall call such a pair a `stationary minimum pair' in this paper. These pairs are interesting because if they exist, then the stationary policy from a pair is not only average-cost optimal for the initial distribution it pairs with, but also, under additional conditions, pathwise optimal for all initial distributions (see Lasserre~\cite{Las99}, Vega-Amaya~\cite{VAm99}, and also the book by Hern\'{a}ndez-Lerma and Lasserre~\cite[Chap.~11.4]{HL99}).

The minimum pair approach we take in this paper was introduced by Kurano~\cite{Kur89} for Borel-space MDPs, motivated by the idea in the earlier work of Borkar~\cite{Bor83,Bor84}, who analyzed average-cost problems based on occupancy measures of the policies.
Unlike the vanishing discount factor approach with which one studies the average cost problem as the limiting case of the discounted problems (cf.~\cite{FKZ12,Sch93,Sen89}), it is a direct approach.
It also differs from a linear programming-based method that searches for a minimum pair among all induced invariant probability measures (cf.~\cite{Den70}). The minimum pair approach is an analysis technique that can be applied to investigate questions such as the existence of stationary optimal policies, for infinite-space MDPs where it is not known a priori if any stationary minimum pair exists or if any stationary policy can induce a stationary Markov chain. The results from this approach can shed light on the optimality properties of average-cost MDPs and moreover, provide the basis for the linear programming approach to studying these problems (cf.~\cite{HL02}).

With the minimum pair approach, Kurano~\cite{Kur89} first considered bounded costs and compact spaces; Hern\'{a}ndez-Lerma \cite{HLe93}, Lasserre~\cite{Las99} and Vega-Amaya \cite{VAm99} subsequently analyzed the case of strictly unbounded costs, with pathwise optimality being the focus of \cite{Las99,VAm99} (see also the books \cite[Chap.~5.7]{HL96}, \cite[Chap.\ 11.4]{HL99}). These prior results all concern lower semicontinuous Borel state and action space MDP models, in which the state transition stochastic kernel is (weakly) continuous and the one-stage cost function is lower semicontinuous. 

Our results are for a countable, discrete action space and strictly unbounded costs. They are analogous to the prior results just mentioned; however, they do not require continuity conditions on the MDP model. Instead we introduce a new majorization condition on the state transition stochastic kernel of the MDP. 
This condition, roughly speaking, requires the existence of finite Borel measures on the state space that can majorize certain sub-stochastic kernels created from the state transition stochastic kernel, at all admissible state-action pairs (see Assumption \ref{cond-pc-3}(M)). Our main idea is to use those majorizing finite measures in combination with Lusin's theorem (see Theorem~\ref{thm-lusin}), which then allows us to extract arbitrarily large sets (large as measured by a given finite measure) on which certain functions involved in our analysis have desired continuity properties.   
By using this technique with the minimum pair approach, we obtain optimality results that can be applied to a class of countable action space MDPs in which, with respect to (w.r.t.)~the state variable, the dynamics and one-stage costs are discontinuous.

The main results of this paper are as follows:
\begin{itemize}[leftmargin=0.45cm,labelwidth=!]
\item We prove the existence of a stationary minimum pair under the new majorization condition (see Assumption~\ref{cond-pc-3}, Prop.~\ref{prp-su-mp}, and Theorem~\ref{thm-su-mp}).
\item We relate the minimum average cost to the limit of the minimum discounted costs as the discount factor vanishes (see Prop.~\ref{prp-malpha}).
\item We also establish other optimality properties of the stationary policy from a stationary minimum pair, in terms of the limit inferior expected average costs and the pathwise average costs (see Theorem~\ref{thm-su-mp2}(a)). These optimality properties are then used, under additional positive Harris recurrence or $f$-regularity conditions on the induced Markov chain, to establish pathwise optimality and strong optimality of the stationary policy for the average-cost MDP (see Theorem~\ref{thm-su-mp2}(b)). 
\end{itemize}
These results can be compared with the minimum pair results for lower semicontinuous models in \cite{HLe93,HL96,HL99,Kur89,Las99,VAm99}. We have also studied the implications of the majorization condition for the linear programming approach and obtained duality results that are comparable with those known in the literature for lower semicontinuous MDP models. These results are reported in a separate paper \cite{Yu-lp19}.
 
Our scope is still limited, however, because with our current proof techniques, we can only handle countable action spaces with the discrete topology. Future research will be to try to extend this work to Borel action spaces and universally measurable policies %\cite{ShrB79}.
(\cite{ShrB79}; \cite[Part II]{bs}).

We remark that Lusin's theorem has been used earlier in a similar way by the author to tackle measurability-related issues in policy iteration for a lower semicontinuous Borel-space MDP model under discounted and total cost criteria~\cite[Sec.~6]{YuB-mvipi}. The average-cost minimum pair problem we address in this paper and the other arguments involved in our analysis are entirely different from those in \cite{YuB-mvipi}, however.

We also mention our related recent work based on the same majorization idea. In~\cite{Yu19} we have introduced another majorization condition to work with the vanishing discount factor approach. As in this paper, that majorization condition is used, instead of the commonly required continuity/compactness conditions on the MDP model, to prove the average cost optimality inequality (ACOI) for MDPs with Borel state and action spaces and universally measurable policies.

The rest of this paper is organized as follows. Section~\ref{sec-2} introduces basic definitions and notations. Section~\ref{sec-3} presents the majorization condition and the main results; it also includes discussion and illustrative examples of the results and the assumptions involved. 
The proofs are given in Section~\ref{sec-4}. Background material about Harris recurrent and regular Markov chains is included in Appendix~\ref{appsec-1}.

\section{Preliminaries} \label{sec-2}
For any metrizable topological space $X$, let $\B(X)$ denote the Borel $\sigma$-algebra and $\P(X)$ the set of probability measures on $\B(X)$, and endow the space $\P(X)$ with the topology of weak convergence.
A \emph{Borel space} (a.k.a.\ standard Borel space) is a separable metrizable topological space that is homeomorphic to a Borel subset of some Polish space (a separable and completely metrizable topological space). If $X$ and $Y$ are Borel spaces, a \emph{Borel measurable stochastic kernel} on $Y$ given $X$, denoted $q(dy \,|\, x)$, is a Borel measurable function from $X$ into $\P(Y)$; equivalently, it is a family of Borel probability measures on $Y$ parametrized by $x$ such that for each $B \in \B(Y)$, the function $q(B \mid \cdot): X \to [0, 1]$ is Borel measurable (see \cite[Def.\ 7.12 and Prop.\ 7.26]{bs}). The stochastic kernel $q(dy \,|\, x)$ is called \emph{continuous} if and only if (iff) it is a continuous function from $X$ into $\P(Y)$ (this case is also called \emph{weak Feller} or \emph{weakly continuous} in the literature).
 
We consider a standard MDP model that has a Borel space $\X$ as its state space and a countable space $\A$ endowed with the discrete topology as its action space.
At a state $x \in \X$, the set of admissible actions is nonempty and denoted by $A(x)$. 
The set-valued map $A: x \mapsto A(x)$ specifies the control constraint in the MDP. 
We assume that its graph $\Gamma := \{(x, a) \mid x \in \X, a \in A(x)\}$ is a Borel subset of $\X \times \A$. 
At a state $x$, taking action $a \in A(x)$ results in a one-stage cost $c(x,a)$ and a probabilistic state transition.
We assume that the state transition is governed by a Borel measurable stochastic kernel $q(dy \,|\, x, a)$ on $\X$ given $\X \times \A$,
and that the one-stage cost function $c: \X \times \A \to [0, +\infty]$ is nonnegative and Borel measurable, finite-valued on $\Gamma$ and taking the value $+\infty$ outside $\Gamma$. 
Later we will impose more conditions on the MDP to study its optimality properties for average cost criteria.

A policy of an MDP consists of a sequence of stochastic kernels on $\A$ that specify for each stage, which admissible actions to take, given the history up to that stage. 
In particular, we consider Borel measurable policies in this paper and often we shall simply call them policies. A \emph{Borel measurable policy} is a sequence $\pi : =(\mu_0, \mu_1, \ldots)$ where for each $n \geq 0$,
$\mu_n\big(da_n \!\mid x_0, a_0, \ldots, a_{n-1}, x_n \big)$ is a Borel measurable stochastic kernel on $\A$ given $(\X \times \A)^{n} \times \X$ and obeys the control constraint of the MDP:
$$ \mu_n\big(A(x_n) \!\mid x_0, a_0, \ldots, a_{n-1}, x_n \big) = 1, \quad \forall \, (x_0, a_0, \ldots, a_{n-1}, x_n) \in (\X \times \A)^n \times \X.$$
(For notational simplicity, although $\A$ is countable, we shall write probability measures on $\A$ in the same way as we do for the possibly uncountably infinite spaces $\X$ and $\X \times \A$.) 
A policy $\pi$ is \emph{stationary} if for all $n \geq 0$, the function $(x_0, a_0, \ldots, a_{n-1}, x_n) \mapsto  \mu_n(d a_n \!\mid\! x_0,  a_0, \ldots, a_{n-1}, x_n)$ depends only on the state $x_n$ and is also independent of the stage $n$. In this case, the policy can be expressed as $\pi = (\mu, \mu, \ldots)$ for a Borel measurable stochastic kernel $\mu(da \,|\, x)$ on $\A$ given $\X$ that obeys the control constraint of the MDP, and we will simply designate this policy by $\mu$. Our results in this paper will center around stationary policies.

Let $\Pi$ denote the space of Borel measurable policies and $\Pi_s$ the subset of all stationary policies.%footnote starts
\footnote{These sets are nonempty since $\A$ is countable and discrete. For example, if we label the elements of $\A$ by $1, 2, \ldots$ and let $A_k = \{ x \in \X \mid k \in A(x) \}$ for $k \geq 1$, then applying action $k$ on the set of states $A_k \setminus \cup_{j < k} A_j$ defines a Borel measurable stationary policy in $\Pi_s \subset \Pi$.}
%footnote ends
We consider infinite horizon problems. A policy $\pi \in \Pi$ and an initial (state) distribution $\zeta \in \P(\X)$ induce a stochastic process $\{(x_n, a_n)\}_{n \geq 0}$ on the product space $(\X \times \A)^\infty$, and the probability measure for this process is uniquely determined by $\zeta$, the sequence of stochastic kernels in $\pi$, and the state transition stochastic kernel $q(dy \,|\, x, a)$ \cite[Prop.~7.28]{bs}. We denote this probability measure by $\Pr^\pi_\zeta$ and the corresponding expectation operator by $\E^\pi_\zeta$. For notational simplicity, when $\zeta$ is the Dirac measure $\delta_x$ concentrated at a single initial state $x$, we shall often write $x$ in place of $\delta_x$.

The $n$-stage expected total cost of $\pi$ for an initial distribution $\zeta$ is given by 
$$ J_n(\pi, \zeta) : = \E^\pi_\zeta \big[ \, \textstyle{\sum_{k=0}^{n-1} c(x_k, a_k)} \, \big],$$
and the \emph{limit superior} and \emph{limit inferior} \emph{(expected) average costs} of $\pi$ are given, respectively, by
$$ J(\pi,\zeta) : = \limsup_{n \to \infty} J_n(\pi, \zeta) /n, \qquad \iJ(\pi,\zeta) : = \liminf_{n \to \infty} J_n(\pi, \zeta) /n.$$ 
We also consider the average costs of $\pi$ along a sample path $(x_0, a_0, x_1, a_1, \ldots)$:
$$ \hat J(\pi, \zeta) : = \limsup_{n \to \infty} \, n^{-1} \sum_{k=0}^{n-1} c(x_k, a_k), \qquad  \underline{\hat{J}}(\pi, \zeta) : = \liminf_{n \to \infty} \, n^{-1} \sum_{k=0}^{n-1} c(x_k, a_k).$$ 
They are nonnegative random variables whose distributions depend on $\pi$ and the initial distribution $\zeta$. 
We will refer to them as the \emph{pathwise average costs} of $\pi$.

We consider several optimality criteria for the average cost problem. The standard notion of optimality is defined w.r.t.\ the limit superior average costs.

\begin{definition}[some notions of average-cost optimality] \label{def-ac-opt} \rm \hfill
\begin{enumerate}[leftmargin=0.65cm,labelwidth=!]
\item[(a)] We call a policy $\pi^*$ \emph{average-cost optimal} iff 
$$ J(\pi^*, x) \leq J(\pi, x), \quad \forall \, \pi \in \Pi, \ x \in \X;$$
and \emph{strongly average-cost optimal} iff 
$$ J(\pi^*, x) \leq \iJ(\pi, x), \quad \forall \, \pi \in \Pi, \ x \in \X.$$
\item[(b)] We call a policy $\pi^*$ \emph{pathwise average-cost optimal} iff for every $\zeta \in \P(\X)$, there exists a constant $\rho_\zeta \geq 0$ such that
$$ \hat J(\pi^*, \zeta) = \rho_\zeta, \qquad \text{$\Pr^{\pi^*}_\zeta$-almost surely}, $$
whereas for every other policy $\pi$, 
$$ \hat J(\pi, \zeta) \geq \rho_\zeta, \qquad \text{$\Pr^\pi_\zeta$-almost surely}. $$
\end{enumerate}
\end{definition}
%\smallskip

With the minimum pair approach, we tackle theoretical questions regarding optimality properties of average-cost MDPs by focusing on the \emph{minimum average cost} over all policies and initial distributions,
$$ \rho^* : = \inf_{\zeta \in \P(\X)} \inf_{\pi \in \Pi} J(\pi, \zeta),$$ 
and those policies that can attain $\rho^*$ for some initial distributions. 
Among such policies, of special interest are stationary policies whose induced Markov chains possess invariant probability measures, for the ergodic theory for general state space Markov chains can then be applied to analyze the induced stationary Markov chains and help in the study of the behavior and optimality properties of such policies. This is the motivation behind the following definitions.

\begin{definition}[minimum pair] \rm 
A pair $(\pi^*, \zeta^*)$ of policy and initial distribution is called a \emph{minimum pair} iff 
$J(\pi^*, \zeta^*) = \rho^*$. 
\end{definition}

{\samepage
\begin{definition}[stationary pair and stationary minimum pair] \rm \hfill
\begin{enumerate}[leftmargin=0.65cm,labelwidth=!]
\item[(a)] For a stationary policy $\mu \in \Pi_s$ and an initial distribution $p \in \P(\X)$, if $p$ is an invariant probability measure of the Markov chain induced by $\mu$ on $\X$, we call $(\mu, p)$ a \emph{stationary pair}. 
The set of all stationary pairs is denoted by $\Delta_s$.
\item[(b)] If $(\mu^*, p^*) \in \Delta_s$ is a minimum pair, we call it a \emph{stationary minimum pair}.
\end{enumerate}
\end{definition}}
%\smallskip

From the ergodic theory for stationary processes and the minimality of $\rho^*$, it is known that if $\rho^* < \infty$ and if a stationary minimum pair $(\mu^*, p^*)$ exists, it satisfies
\begin{equation} \label{eq-opt-spair}
       J(\mu^*, x)  = \iJ(\mu^*, x) = \rho^*, \qquad \text{for $p^*$-almost all } x \in \X
\end{equation}
(see \cite[Prop.~11.4.4]{HL99}). We will focus primarily on the existence question and other optimality properties of a stationary minimum pair.

The stationary policy $\mu^*$ in a stationary minimum pair need not be optimal w.r.t.\ any one of the average-cost optimality criteria in Definition~\ref{def-ac-opt}, as can be seen by comparing Definition~\ref{def-ac-opt} with (\ref{eq-opt-spair}). Also, the set of initial states for which $\mu^*$ attains the minimum average cost $\rho^*$ can be small, if the support of $p^*$ is small. To ensure that $\mu^*$ is optimal for all initial states, it is not enough that the induced Markov chain has invariant probability measures. We will need the Markov chain to have stronger ergodic properties. 
So, near the end of our analysis, we will consider two classes of Markov chains, positive Harris recurrent and $f$-regular Markov chains (see Appendix~\ref{appsec-1} for their definitions). We will use them as conditions on $\mu^*$ and combine their ergodic properties with other properties of a stationary minimum pair to obtain the average-cost optimality of $\mu^*$.

\section{Main Results: Existence of Stationary Minimum Pair and its Optimality Properties} \label{sec-3}

We start with the main assumption for our results.
Recall that $\Gamma = \big\{ (x, a) \mid x \in \X, a \in A(x)\big\}$ is the graph of the control constraint of an MDP. If $B \subset \X \times \A$, let $\proj_\X(B)$ (resp.~$\proj_\A(B)$) 
denote the projection of $B$ on $\X$ (resp.~$\A$). 
The complement of a set $B$ in some space is denoted by $B^c$.

\begin{assumption}
\label{cond-pc-3} \hfill
\begin{enumerate}[leftmargin=0.85cm]
\item[\rm (G)] For some $\pi \in \Pi$ and $\zeta \in \P(\X)$, the average cost $J(\pi,\zeta) < \infty$.
\item[\rm (SU)] There exists a nondecreasing sequence of compact sets $\Gamma_j \uparrow \Gamma$ such that
$$ \lim_{n \to \infty} \inf_{(x, a) \in \Gamma_j^c} c(x, a) = + \infty.$$
\item[\rm (M)] For each compact set $K \in \{\proj_\X(\Gamma_j)\}$, there exist an open set $O \supset K$, a closed set $D \subset \X$, and a finite measure $\nu$ on $\B(\X)$ (all of which can depend on $K$) such that
\begin{equation}
     q\big((O \setminus D) \cap B  \mid x, a \big) \leq \nu(B), \qquad \forall \, B \in \B(\X), \ (x, a) \in \Gamma, \label{eq-maj-minpair}
\end{equation}     
where the closed set $D$ (possibly empty) is such that restricted to $D \times \A$, the state transition stochastic kernel $q(dy\,|\, x, a)$ is continuous and the one-stage cost function $c$ is lower semicontinuous. 
\end{enumerate}
\end{assumption}
%\smallskip

The conditions (G) and (SU) are standard (see \cite[Assumption 11.4.1(a), (c)]{HL99}).
The condition (G) is to exclude vacuous problems. The condition (SU) defines what we mean by strictly unbounded costs. 
It implies the additional constraint on the spaces of the MDPs we consider: $\Gamma =\cup_j \Gamma_j$ must be $\sigma$-compact and hence $\X$ must also be $\sigma$-compact. 
(See \cite[Remark 11.4.2(a3)]{HL99} for a sufficient condition for (SU) that involves an upper-semicontinuous map $A(\cdot)$.)

The condition (M) is the new majorization condition we introduce.
In this condition, roughly speaking, we divide the state space into two parts, a closed set $D$ on which the model has nice continuity properties, and the complement set $D^c$ on which we impose a majorization condition (\ref{eq-maj-minpair}). The condition (M) is satisfied trivially by letting $D = \X$, if the entire model is lower semicontinuous (this means, in the discrete action setting considered here, that for each $a \in \A$, $q(dy \,|\, x, a)$ is continuous in $x$ and $c(x,a)$ is lower semicontinuous in $x$). 

For discontinuous models in general, the condition (M) seems natural in cases where the probability measures $\{q(\cdot \mid x, a) \mid (x, a) \in \Gamma\}$ have densities on the complement set $D^c = \X \setminus D$ w.r.t.\ a common ($\sigma$-finite) reference measure. In such cases, under practical conditions on those density functions, the condition (M) holds; see Example~\ref{ex-2} in Section~\ref{sec-minpair-ex} for an illustration.

We defer a further discussion about the conditions (M) and (SU) to Section~\ref{sec-minpair-ex}.

\subsection{Results} \label{sec-3.1}
We state the main results in this subsection. 

Proposition~\ref{prp-su-mp} and Theorem~\ref{thm-su-mp} establish the existence of a stationary minimum pair. They are analogous to the prior results in \cite{HL96,Kur89} for lower semicontinuous models; in particular, they can be compared with \cite[Theorems 2.1 and 2.2]{Kur89} and \cite[Lemma 5.7.10 and Theorem 5.7.9(a)]{HL96}.

%\smallskip
\begin{prop} \label{prp-su-mp}
Under Assumption~\ref{cond-pc-3}, for any pair $(\pi, \zeta) \in \Pi \times \P(\X)$ with $J(\pi, \zeta) < \infty$, there exists a stationary pair $(\bar \mu, \bar p) \in \Delta_s$ with $J(\bar \mu, \bar p) \leq J(\pi, \zeta)$.
\end{prop}

\begin{theorem} \label{thm-su-mp}
Under Assumption~\ref{cond-pc-3}, there is a minimum pair $(\mu^*, p^*) \in \Delta_s$. 
\end{theorem}
%\smallskip

Proposition~\ref{prp-malpha} and Theorem~\ref{thm-su-mp2} below are analogous to the prior results in \cite{HL99,Las99,VAm99} for lower semicontinuous models. In particular, Prop.~\ref{prp-malpha} can be compared with \cite[(13) in Theorem 3.6(a)]{VAm99}. Theorem~\ref{thm-su-mp2} (except for its last part regarding the strong optimality) can be compared with \cite[Theorems 3.4 and 3.6(b)]{VAm99} and \cite[Theorem 11.4.6(a) and (c)]{HL99}, as well as with \cite[Theorem III.1]{Las99}. 

Proposition~\ref{prp-malpha} shows a relation between the minimum average cost $\rho^*$ and the minimum $\alpha$-discounted costs $m_\alpha$, which are defined as follows. 
For $\alpha \in (0,1)$, the $\alpha$-discounted expected cost of a policy $\pi$ for an initial distribution $\zeta$ is given by
$$ v_\alpha(\pi, \zeta) : =  \textstyle{ \E^\pi_\zeta \big[ \sum_{n=0}^\infty \alpha^n \, c(x_n, a_n) \big]}.$$
Let
\begin{equation}
  m_\alpha := \inf_{x \in \X} \inf_{\pi \in \Pi} v_\alpha(\pi, x) = \inf_{\zeta  \in \P(\X)} \inf_{\pi \in \Pi} v_\alpha(\pi, \zeta),
\end{equation}  
where the second equality is easy to verify. By a Tauberian theorem (see e.g., \cite[Lemma 8.10.6]{Puterman94}),
$$  \liminf_{\alpha \uparrow 1} \, (1 - \alpha)\, m_\alpha \leq  \limsup_{\alpha \uparrow 1} \, (1 - \alpha) \, m_\alpha \leq \rho^*.$$
Proposition~\ref{prp-malpha} asserts the equality of the above three quantities.
Its proof will also show, like in \cite{VAm99}, that a stationary minimum pair can also be constructed from nearly optimal policies of a sequence of discounted problems with vanishing discount factors. 

%\smallskip
\begin{prop} \label{prp-malpha}
Under Assumption~\ref{cond-pc-3}, $\lim_{\alpha \uparrow 1} \, (1 - \alpha) \, m_\alpha = \rho^*$.
\end{prop}
%\smallskip

Theorem~\ref{thm-su-mp2} below extends Theorem~\ref{thm-su-mp} by considering optimality properties that involve the limit inferior expected average costs and pathwise average costs of policies. Its part (a), in particular (\ref{eq-thm2-a2}), is the key result that leads to the existence of a pathwise optimal stationary policy given in the part (b).

The assumptions of Theorem~\ref{thm-su-mp2}(b) involve positive Harris recurrent and $f$-regular Markov chains (see \cite{MeT09} or our Appendix~\ref{appsec-1} for the definitions of these Markov chains). For this part of the theorem, we define an expected one-stage cost function $c_\mu$ for a stationary policy $\mu \in \Pi_s$ by
$$c_{\mu}(x) : = \textstyle{\int_\A c(x, a) \, \mu(da \mid x)}, \qquad x \in \X.$$

{\samepage
\begin{theorem} \label{thm-su-mp2}
Let Assumption~\ref{cond-pc-3} hold.
\begin{enumerate}[leftmargin=0.65cm,labelwidth=!]
\item[\rm (a)] For any policy $\pi$ and initial distribution $\zeta$, the limit inferior average cost satisfies
\begin{equation} \label{eq-thm2-a1}
 \iJ(\pi, \zeta) \geq \rho^*,
\end{equation} 
and the pathwise average cost satisfies
\begin{equation} \label{eq-thm2-a2}
  \liminf_{n \to \infty} \, n^{-1} \textstyle{ \sum_{k=0}^{n-1} c(x_k, a_k)} \geq \rho^*, \qquad \text{$\Pr^{\pi}_\zeta$-almost surely}.
\end{equation}  
\item[\rm (b)] Let $\mu^*$ 
be the policy in a stationary minimum pair.
If $\mu^*$ induces a positive Harris recurrent Markov chain on $\X$, then $\mu^*$ is 
pathwise average-cost optimal with 
\begin{equation} \label{eq-thm2-a3}
    \hat J(\mu^*, \zeta) = \underline{\hat{J}}(\mu^*, \zeta) = \rho^*, \ \ \  \text{$\Pr^{\mu^*}_\zeta$-almost surely},  \qquad \forall \, \zeta \in \P(\X).
\end{equation} 
If, in addition, $c_{\mu^*}$ is finite-valued and the induced Markov chain is $f$-regular for $f=c_{\mu^*} + 1$, then $\mu^*$ is also strongly average-cost optimal with 
\begin{equation} \label{eq-thm2-a4}
 J(\mu^*,x) = \iJ(\mu^*, x) = \rho^*, \qquad \forall \, x \in \X.
\end{equation} 
\end{enumerate}
\end{theorem}
}
%\smallskip

\begin{myremark} \rm \label{rmk-pathwise-opt}
For a stationary minimum pair $(\mu^*, p^*)$, without the additional recurrence/regularity assumptions in Theorem~\ref{thm-su-mp2}(b), we can only assert that
\begin{equation} \label{eq-thm2-a3-alt}
 \hat J(\mu^*, x) = \underline{\hat{J}}(\mu^*, x) = \rho^*,  \ \ \ \text{$\Pr^{\mu^*}_x$-almost surely},  \quad \text{$\forall \, x \in \X \,$ s.t.} \ J(\mu^*, x) = \rho^*,
\end{equation} 
which, in view of the property (\ref{eq-opt-spair}), implies that 
$$\hat J(\mu^*, x) = \underline{\hat{J}}(\mu^*, x) = \rho^*, \qquad \text{for $p^*$-almost all } x \in \X.$$
The conclusion (\ref{eq-thm2-a3-alt}) follows from Theorem~\ref{thm-su-mp2}(a); the reasoning is the same as that used in the proofs of \cite[Lemma III.1]{Las99} and \cite[Theorem 11.4.6(b)]{HL99}. For an example where (\ref{eq-thm2-a3-alt}) holds but (\ref{eq-thm2-a3}) does not, or where (\ref{eq-opt-spair}) holds but (\ref{eq-thm2-a4}) does not, see Example~\ref{ex-1} below.
\myqed
\end{myremark}

We prove the above results in Section~\ref{sec-4}. In the rest of this section, we discuss some aspects of them and give illustrative examples.

\subsection{Discussions and Illustrative Examples} \label{sec-minpair-ex}

As we mentioned in the introduction, besides the minimum pair approach, another method to study average-cost MDPs is the vanishing discount factor approach. With this approach, one aims to first establish the ACOI (average cost optimality inequality) for an MDP, and then infer from the ACOI the existence of an optimal or nearly optimal stationary policy (see e.g., \cite{FKZ12,Sch93,Sen89}). It is known that the main assumptions used by the two approaches do not imply each other;
a discussion about this and a finite state and action example are given in \cite[p.~121 and Example 5.7.3, p.~114]{HL96}.
We give now another brief discussion and a countable state example to compare the two approaches in the context of this paper.

While we do not require the continuity of $c(x,a)$ and $q(dy \,|\, x,a)$ in $(x,a)$, because the action space $\A$ is discrete, trivially, for each state $x \in \X$ and w.r.t.\ the actions, $c(x,\cdot)$ is continuous and $q(dy \,|\, x,\cdot)$ strongly continuous (i.e., continuous w.r.t.\ set-wise convergence). Also, under the strictly unbounded cost assumption (SU), we have a compact (finite) level set $\{ a \in A(x) \,|\, c(x,a) \leq r \}$ for every state $x$ and $r > 0$. Thus the MDPs we consider actually satisfy a type of continuity and compactness condition that has been used to establish 
the ACOI via the vanishing discount factor approach (see e.g., \cite[Condition (S)]{Sch93} and \cite[Assumption 2.1]{HLe91}). 
Through the ACOI, one can also obtain the existence of a stationary average-cost optimal policy, which is comparable to our results in Theorem~\ref{thm-su-mp2}(b) from the minimum pair approach.
However, those ACOI results typically involve also a pointwise upper bound condition on the relative 
value functions $v_\alpha^*(x) - m_\alpha$ of the $\alpha$-discounted problems, 
where $v_\alpha^*$ is the value function given by $v_\alpha^*(x) : = \inf_{\pi \in \Pi} v_\alpha(\pi, x), x \in \X$ (recall $m_\alpha = \inf_{x \in \X} v_\alpha^*(x)$ by definition).
Specifically, the condition (B) \cite{Sch93} or the weaker (\underline{B}) \cite{FKZ12} is required: for every $x \in \X$, 
\begin{equation}
\text{(B):} \quad \sup_{\alpha \in (0,1)} \big( v^*_\alpha(x) - m_\alpha \big) < \infty, \qquad \text{(\underline{B}):}  \quad \liminf_{\alpha \uparrow 1} \, \big( v^*_\alpha(x) - m_\alpha \big) < \infty. \notag
\end{equation}

Example~\ref{ex-1} below shows that these conditions need not be satisfied by the MDP model under our assumptions, even when the positive Harris recurrence condition in Theorem~\ref{thm-su-mp2}(b) holds. 
In addition, we also use this example to illustrate the importance of the Harris recurrence and $f$-regularity condition in Theorem~\ref{thm-su-mp2}(b).

\begin{example} \label{ex-1} \rm
We let the MDP be an uncontrolled countable-space Markov chain discussed in~\cite[Chap.\ 11.1, p.\ 259]{MeT09}, choosing its parameters to make it a positive Harris recurrent Markov chain that is not regular (see \cite{MeT09} or our Appendix~\ref{appsec-1} for these concepts about Markov chains).
The states are $\{0, 1, 2, \ldots\}$. The probability of transitioning from state $i$ to $j$, denoted by $P(i,j)$, is given by
$$ P(0,0) = 1 \ \ \ \text{and} \ \ \  P(i,0) = \beta_i> 0, \quad P(i, i+1) = 1 - \beta_i, \quad \forall \, i \geq 1.$$ 
Suppose the probabilities $\{\beta_i\}$ are such that
\begin{equation} \label{eq-nonregular-mc}
 \textstyle{ \prod_{i=1}^\infty ( 1 - \beta_i) = 0, \qquad \sum_{k = 1}^\infty \prod_{i=1}^k (1 - \beta_i) = \infty.}
\end{equation}  
By the first relation above, this Markov chain is positive Harris recurrent, with $0$ being the only recurrent state.
By the second relation above, from an initial state $i \not=0$, the expected time to hit state $0$ is infinite: with $\tau_{\{0\}}= \inf \{ n \geq 1 \mid x_n = 0 \}$,
$\E_i \big[ \tau_{\{0\}} \big] = + \infty$. Thus the Markov chain is not regular.

If we let the one-stage costs be $c(0) = 0$, $c(i) = 1$ for $i \geq 1$, we have $m_\alpha = 0$ for all $\alpha \in (0,1)$, whereas from any state $i \geq 1$, 
$v^*_\alpha(i) \uparrow \E_i \big[ \tau_{\{0\}} \big] = +\infty$ as $\alpha \uparrow 1$.
So the conditions (B) and (\underline{B}) are violated. Now if we let $c(i) = i$ to make $c(\cdot)$ strictly unbounded, then $v^*_\alpha(i)$ is even larger than in the previous case and hence (B) and (\underline{B}) cannot be satisfied. Moreover, as the Markov chain is not regular, by~\cite[Theorem 11.3.11(i)]{MeT09} (with the petite set being $\{0\}$ and the constant $b = 1$ in the drift condition therein), the ACOI in this case, 
$$0 + h(i) \geq c(i) + \beta_i h(0) + (1 - \beta_i) h(i+1), \quad i \geq 0$$
(where $\beta_0:=1$) does not admit a nonnegative finite-valued solution $h(\cdot)$.

On the other hand, Assumption~\ref{cond-pc-3} holds trivially in this example. 
The only policy in this MDP is trivially strongly optimal, and the pathwise average costs satisfy $\hat J(i) = \rho^* = 0$ almost surely for all states~\cite[Theorem 17.1.7]{MeT09}.

In contrast to this pathwise optimality, however, one can choose $\beta_i$ and $c(i), i \geq 1$ so that the average cost $J(i) > \rho^*=0$ for $i \geq 1$ (in particular, if $\beta_i = 1/(i+1)$ and $c(i) = (i+1)$, then $J(i) = i$).
This shows that, in general, for the equality (\ref{eq-thm2-a4}) in Theorem~\ref{thm-su-mp2}(b) to hold, that is, $J(\mu^*,x) = \iJ(\mu^*, x) = \rho^*$ for all $x \in \X$, the positive Harris recurrence assumption alone is insufficient.  

Now if, instead of (\ref{eq-nonregular-mc}), we let $\prod_{i=1}^\infty ( 1 - \beta_i) > 0$ (e.g., let $\beta_i = 1 - (1 + \tfrac{2}{i+1})/(1+ \tfrac{2}{i})$ to have $\prod_{i=1}^\infty ( 1 - \beta_i) = 1/3$), then the Markov chain is ($\psi$-irreducible and) positive recurrent but not positive Harris recurrent. In this case, with strictly unbounded $c(\cdot)$, for every initial state $i \not= 0$, it occurs with probability $\prod_{j=i}^\infty ( 1 - \beta_j) > 0$ that the pathwise costs $\hat J(i) = \underline{\hat{J}}(i) = +\infty$. This shows that without the positive Harris recurrence assumption, the property (\ref{eq-thm2-a3}) in Theorem~\ref{thm-su-mp2}(b) need not hold in general and we have only (\ref{eq-thm2-a3-alt}) instead.
\myqed
\end{example}

\begin{myremark} \rm It is also possible to establish a partial form of ACOI without using the upper bound condition (B) or (\underline{B}), as Hern\'{a}ndez-Lerma and Lasserre \cite{HL94} showed. Specifically, under a set of assumptions on the MDP model and the value functions $v^*_\alpha$ of the discounted problems, they proved that ACOI holds for a nonempty subset of states and there is a stationary and nonrandomized policy that attains the minimum average cost $\rho^*$ for all initial states in that subset. This is similar to the property (\ref{eq-opt-spair}) we get from the existence of a stationary minimum pair given by Theorem~\ref{thm-su-mp}.

The countable action space MDP we consider here satisfies the model assumption in \cite{HL94} (which is the same as  \cite[Assumption 2.1]{HLe91} mentioned earlier). Other requirements in \cite{HL94} include: 
(i) $J(\pi, \bar x) < \infty$ for some policy $\pi$ and state $\bar x$; and (ii) there exist $N \geq 0$ and $\bar \alpha \in (0,1)$, such that $v^*_\alpha(x) - v^*_\alpha(\bar x) \geq -N$ for all $\alpha \in (\bar \alpha, 1)$. 
The proof in \cite{HL94} shows that such a state $\bar x$ must be a state with the minimum average cost $\rho^*$. Thus, although the requirement (i) looks almost the same as our Assumption~\ref{cond-pc-3}(G), the two are essentially different in nature. Identifying a state $\bar x$ that meets the conditions in \cite{HL94} would generally be much harder than verifying (G). This and the preceding discussions serve to show that even for the simpler, countable action space MDPs, the vanishing discount factor approach and the minimum pair approach are genuinely different and complementary to each other, and our results are not subsumed by the existing results from the former approach.
\myqed
\end{myremark}

We use the next Example~\ref{ex-2} to demonstrate a case where the majorization condition (M) is satisfied naturally. For simplicity, we consider a problem similar to a one-dimensional linear-quadratic (LQ) control problem but with a discretized action space and ``modulated'' quadratic costs. The same reasoning can be applied to higher dimensional problems with nonlinear dynamics and additive noise.

\begin{example} \label{ex-2} \rm
 Let $\X = \R$, $\A \subset \R$, and $c(x,a) = \beta(x) \, ( x^2 +  a^2)$, where $\beta(\cdot)$ is an arbitrary (measurable) nonnegative function such that $\lim\inf_{|x| \to \infty} \beta(x) > 0$, $\sup_{x \in \R} \beta(x) < \infty$.
Let 
$$x_{n+1} = x_{n} + a_n + \omega_n(x_n, a_n), \quad n \geq 0,$$
where $\omega_n(x_n, a_n)$ is a random disturbance whose distribution, given $\{(x_k, a_k)\}_{k \leq n}$, 
depends only on $(x_n, a_n)$ and is given by $F_{x,a} \in \P(\R)$ for $(x_n, a_n)=(x,a)$.
Consider a discrete action space $\A = \big\{ k \delta \mid k = 0, \, \pm 1, \, \pm 2, \, \ldots \big\}$ for some small $\delta > 0$.
For simplicity, suppose $\Gamma = \X \times \A$. The function $c(\cdot)$ is clearly strictly unbounded; e.g., we can let $\Gamma_j = [-j, j] \times \big\{ k \delta \mid -j \leq k \leq j\}$ in Assumption~\ref{cond-pc-3}(SU).

Now consider Assumption~\ref{cond-pc-3}(M).
Suppose that for all $(x,a) \in \Gamma$, w.r.t.\ the Lebesgue measure, the distributions $F_{x,a}$ have densities $f_{x,a}$ that are bounded uniformly from above by $\ell$.
For a closed interval $K = [-j, j]$, consider the open interval $O = (-j-1, j+1)\supset K$. 
A finite measure $\nu$ satisfying Assumption~\ref{cond-pc-3}(M) (with the closed set $D = \emptyset$) is simply given by $\ell$ times the Lebesgue measure on~$O$. 

Regarding Assumption~\ref{cond-pc-3}(G), suppose that $F_{x,a}$, $(x, a) \in \Gamma$, have zero means and variances bounded uniformly by $\sigma^2$. Then a policy $\pi$ that satisfies Assumption~\ref{cond-pc-3}(G) for the initial state $x=0$ is the one that chooses the action $a_n = \argmin_{a \in \A, |a| \leq |x_n|} | x_n + a|$ (since $J(\pi, 0) \leq 2  (\delta^2 + \sigma^2) \cdot \sup_{x \in \R} \beta(x) < \infty$, as can be verified). In this simple example, we can see immediately a solution to the condition (G). For more complicated problems, Markov chain theory can be useful in finding a stationary policy $\pi$ with a finite average cost; see \cite[Sec.~IV.A]{Mey97}.

Thus, except for Theorem~\ref{thm-su-mp2}(b), which requires additional recurrence/regularity conditions, all the theorems we gave hold in this example. 
There is no need here for the continuity of $q(dy \mid x, a)$ or $c(x,a)$ in $x$. 
\myqed
\end{example}

In the rest of this section, we discuss some limitations in the conditions (M) and (SU), when dealing with discontinuous MDP models. 

\begin{myremark}[about (M) and the set $D$] \rm
As demonstrated in Example~\ref{ex-2}, the majorization condition (M) seems practical when $\{q(dy\,|\,x,a) \mid (x,a) \in \Gamma\}$ have densities w.r.t.\ a common $\sigma$-finite measure. When those probability measures have (nontrivial) purely atomic components and the state space $\X$ is uncountable, the majorization condition (M) can fail, even if the state transition stochastic kernel $q(dy\,|\,x,a)$ is continuous at all but one point.
For example, suppose $\X = [0,2]$ and $\A =\{0\}$, and $q(dy\,|\,x,0) = \delta_{x}(dy)$ for $x < 1$, $q(dy\,|\,x,0) = \delta_{x/2}(dy)$ for $x \geq 1$.
Then discontinuity occurs at $x=1$ only. Let the closed set $D$ in the condition (M) be $\X \setminus (1-\epsilon, 1)$ for an arbitrarily small $\epsilon > 0$. But no finite measure can satisfy the inequality (\ref{eq-maj-minpair}) for $O = \X$ and the chosen $D$.
By excluding the set $D$ from $O$ in the inequality (\ref{eq-maj-minpair}), our objective is to broaden the range of applicability of the majorization argument. But in situations like the above, however we choose $D$, it does not help.\myqed
\end{myremark}

\begin{myremark}[about (SU)] \rm
We can use the majorization condition to handle certain types of discontinuities in $q(dy|\cdot, a)$ and $c(\cdot,a)$. 
However, for physical systems, it is also natural to have discontinuities in the control constraint $A(\cdot)$. 
For example, let $\X = [0,2]$, $\A=\{0,1\}$, and $A(x) = \{0,1\}$ for $x < 1$ and $A(x) = \{0\}$ for $x \geq 1$. 
Then the set-valued map $A(\cdot)$ is discontinuous at $x = 1$. To satisfy (SU), one must have the cost $c(x,1) \uparrow +\infty$ as $x \uparrow 1$, which is an unnatural requirement. This kind of discontinuity due to $A(\cdot)$ is hard to handle by the techniques in this paper, since we also rely on the condition (SU) in obtaining a tight family of probably measures on $\Gamma$ to start our analyses (see Section \ref{sec-4}). 
(The set-valued map $A(\cdot)$ here also violates the upper-semicontinuity condition discussed in \cite[Remark 11.4.2(a3)]{HL99}.)
\myqed
\end{myremark}

\section{Proofs} \label{sec-4}
We now prove the results given in Section~\ref{sec-3}.
The broad proof steps will be similar to those in the prior work for lower semicontinuous models \cite{HLe93,HL99,Kur89,VAm99}. But 
since the MDP model here is not lower semicontinuous, the arguments to carry out some of the steps are different.
We shall focus on those steps in our proofs.

The tool we will need to work with the majorization condition (M) is:

\begin{theorem}[Lusin's theorem {\cite[Theorem 7.5.2]{Dud02}}]  \label{thm-lusin}
Let $X$ be any topological space and $\nu$ a finite, closed regular%footnote starts
\footnote{A finite Borel measure $\nu$ on $X$ is \emph{closed regular} if for every Borel subset $B$ of $X$, $\nu(B) =  \sup \big\{ \nu(F) \mid F \ \text{closed}, F \subset B \big\}$ \cite[p.~224]{Dud02}. On a metric space, any finite Borel measure is closed regular \cite[Theorem 7.1.3]{Dud02}.}
%footnote ends
Borel measure on $X$. Let $S$ be a separable metric space and let $f$ be a Borel measurable function from $X$ into $S$. Then for any $\epsilon > 0$ there is a closed set $F \subset X$ such that $\nu(X \setminus F) < \epsilon$ and the restriction of $f$ to $F$ is continuous.
\end{theorem}

We will apply Lusin's theorem to the one-stage cost function $c(\cdot,\cdot)$ on $\X \times \A$ and also to the state transition stochastic kernel, i.e., the $\P(\X)$-valued function $q(dy \,|\, \cdot, \cdot)$ on $\X \times \A$, with $\nu$ being a finite Borel measure from Assumption~\ref{cond-pc-3}(M). 

Before proceeding, let us recall a few facts regarding probability measures on metric spaces that will be needed. 
Recall that on any metric space $X$, a family $\mathcal{E} \subset \P(X)$ is called \emph{tight} iff for any $\epsilon > 0$, there exists a compact set $K \subset X$ such that $p(K) > 1 - \epsilon$ for all probability measures $p \in \mathcal{E}$. 
Prohorov's theorem asserts that for a tight family $\mathcal{E}$, any sequence in $\mathcal{E}$ has a further subsequence converging weakly to a Borel probability measure on $X$ \cite[Theorem 6.1]{Bil68}.

Let $\C_b(X)$ denote the set of (real-valued) bounded continuous functions on $X$. 
If a sequence $\{p_n \}$ in $\P(X)$ converges weakly to $p \in \P(X)$, we shall write $p_n \wto p$.
Recall that by definition, 
$p_n \wto p \in \P(X)$ iff $\int f \, d p_n  \to  \int f dp$ for all $f \in \C_b(X)$,
and by \cite[Prop.\ 11.3.2]{Dud02}, two probability measures $p, p' \in \P(X)$ are equal iff $\int f \, dp = \int f \, dp'$ for all $f \in \C_b(X)$. 
If the metric space $X$ is separable, then by \cite[Chap.~II, Theorem~6.6]{Par67}, there exists a countable set $\{f_1, f_2, \ldots\} \subset \C_b(X)$ such that $p_n \wto p \in \P(X)$ iff
\begin{equation} 
   \textstyle{ 
\int f_\ell \, d p_n \, \to \, \int f_\ell \, dp}, \qquad \forall \, \ell \geq 1, \notag
\end{equation}
and this also means, by \cite[Prop.\ 11.3.2]{Dud02}, that for any $p,p' \in \P(X)$,
\begin{equation} \label{eq-detclass}
  p = p' \qquad \Longleftrightarrow \qquad  \textstyle{\int f_\ell \, d p \, = \, \int f_\ell \, dp'}, \qquad \forall \, \ell \geq 1.
\end{equation}   
The relation (\ref{eq-detclass}) will be important in analyzing the pathwise average costs of policies.

\subsection{Proofs of Prop.~\ref{prp-su-mp} and Theorem~\ref{thm-su-mp}}
Our line of reasoning is the same as that for \cite[Theorem 5.7.9(a) and Lemma 5.7.10]{HL96}. 
Consider the process $\{(x_n, a_n)\}$ induced by a policy $\pi$ and initial distribution $\zeta$ with the average cost $J(\pi,\zeta) < \infty$ as assumed in Prop.~\ref{prp-su-mp}. 
Let $\gamma_n \in \P(\X \times \A)$ be the marginal distribution of $(x_n, a_n)$, and 
define $\bar \gamma_n \in  \P(\X \times \A)$, $n \geq 1$, to be the averages
$$ \textstyle{\bar \gamma_n : = \frac{1}{n} \sum_{k=1}^n \gamma_k.}$$
Since $\pi$ obeys the control constraint of the MDP, all these probability measures $\gamma_n, \bar \gamma_n$ are concentrated on the set $\Gamma$, so their restrictions to $\Gamma$ are in $\P(\Gamma)$. We shall use the notation $\gamma_n^\Gamma, \bar \gamma_n^\Gamma$ for their restrictions to $\Gamma$.

The first step is to extract a weakly convergent subsequence from $\{\bar \gamma_n\}$ by using Assumption~\ref{cond-pc-3}(SU).
For later use, let us choose a subsequence $\{\bar \gamma_{n_k}\}$ such that 
$$  \textstyle{ \lim_{k \to \infty} \int c \, d \bar \gamma_{n_k}  =  \liminf_{n \to \infty}  \int c \, d \bar \gamma_n.}
 $$
The assumption $J(\pi, \zeta) < \infty$ implies not only that $\limsup_{n \to \infty} \int c \, d \bar \gamma_n = J(\pi, \zeta)  < \infty$ but also that $\sup_{n} \int c \, d \bar \gamma_n  < \infty$.
Since $c$ is strictly unbounded by Assumption~\ref{cond-pc-3}(SU), this implies that the family $\{\bar \gamma_n^\Gamma\}$ is tight in $\P(\Gamma)$. Then by Prohorov's Theorem~\cite[Theorem 6.1]{Bil68}, the subsequence $\{\bar \gamma_{n_k}^\Gamma\}$ has a further subsequence that converges weakly to some $\bar{\gamma}^\Gamma \in \P(\Gamma)$. We shall denote that further subsequence also by $\{\bar{\gamma}_{n_k}^\Gamma\}$ to simplify notation. Obviously, $\bar \gamma^\Gamma$ can be extended to a Borel probability measure $\bar \gamma$ on $\X \times \A$ with $\bar \gamma(\Gamma) = 1$.

By \cite[Cor.\ 7.26.1 and Cor.\ 7.27.2]{bs}, we can decompose $\bar \gamma$ into its marginal $\bar p$ on $\X$ and a Borel measurable stochastic kernel $\bar \mu(da  \,|\, x)$ on $\A$ given $\X$, and if necessary, by modifying $\bar \mu(da \,|\, x)$ at a set of $x$ with $\bar p$-measure zero, we can make it obey the control constraint: 
$$ \bar \gamma (d(x,a)) =  \bar \mu( da \mid x) \, \bar p(dx),  \qquad \text{and}  \qquad \bar \mu(A(x) \mid x) = 1,  \quad \forall \, x \in \X.$$
This gives us a stationary policy $\bar \mu$ and a probability measure $\bar p \in \P(\X)$.

In order to prove Prop.~\ref{prp-su-mp}, we need to show that (i) $(\bar \mu, \bar p)$ is a stationary pair and (ii) $J(\bar \mu, \bar p) \leq J(\pi, \zeta)$. More specifically:
\begin{itemize}[leftmargin=0.8cm]
\item[(i)] To show that $(\bar \mu, \bar p)$ is a stationary pair, we need to show that $\bar p$ is an invariant probability measure of the Markov chain $\{x_n\}$ induced by $\bar \mu$:
$$ \bar p(B) = \int_{\X} \int_{\A} q(B \mid x, a) \, \bar \mu(da \mid x) \, \bar p(dx), \qquad \forall \, B \in \B(\X). $$
By \cite[Prop.\ 11.3.2]{Dud02}, it amounts to showing that for every $v \in \C_b(\X)$,
\begin{equation}\label{eq-prf-su2}
   \int_{\X \times \A} \int_{\X} v(y) \, q(dy \mid x, a) \, \bar \gamma (d(x,a)) = \int_{\X} v(x) \, \bar p(dx).
\end{equation}
\item[(ii)] If (i) is proved, then $J(\bar \mu, \bar p) = \int c \, d \bar \gamma$ by the invariance property of $\bar p$, so to prove the desired relation $J(\pi, \zeta) \geq J(\bar \mu, \bar p)$, we also need to show that 
$$  \limsup_{n \to \infty} \textstyle{ \int c \, d \bar \gamma_n \geq \int c \, d \bar \gamma.}$$
For later use, we will instead prove the stronger inequality
\begin{equation} \label{eq-prf-su1}
 \liminf_{n \to \infty}  \int c \, d \bar \gamma_n = \lim_{k \to \infty} \int c \, d \bar \gamma_{n_k}  \geq \int c \, d \bar \gamma.
\end{equation}
\end{itemize}
To prove (\ref{eq-prf-su2}) and (\ref{eq-prf-su1}) (which, in the previous works, were proved by using the lower semicontinuous model assumption), we shall make use of Lusin's theorem and the following implication of Assumption~\ref{cond-pc-3}(M). 

Let $\bar p_n$ denote the marginal of $\bar \gamma_n$ on $\X$. Recall that $\bar p$ is the marginal of $\bar \gamma$ on $\X$.

\begin{lemma} \label{lem-su-mp}
Let the open set $O$, the closed set $D$, and the finite measure $\nu$ on $\B(\X)$ be as in Assumption~\ref{cond-pc-3}(M) for some $K \in \{\proj_\X(\Gamma_j)\}$. Then for all $B \in \B(\X)$,
$$ \bar p \big(  (O \setminus D) \cap B  \big) \leq \nu(B), \qquad  \bar p_n \big( (O \setminus D) \cap B \big) \leq \nu(B), \quad  \forall \, n \geq 1.$$
\end{lemma}

\begin{proof}
For $n \geq 1$, consider the marginal distribution $\gamma_n$ of $(x_n, a_n)$. For any $E \in \B(\X)$, $\gamma_n(E \times \A) = \int q(E \mid x, a) \, \gamma_{n-1}(d(x,a))$, so by Assumption~\ref{cond-pc-3}(M), for any $B \in \B(\X)$, $\gamma_n( \{(O \setminus D) \cap B \}  \times \A) \leq \nu(B)$. Since $\bar p_n(\cdot) =  \bar \gamma_n(\cdot \times \A) = \tfrac{1}{n} \sum_{k=1}^n \gamma_k( \cdot \times \A)$, the desired inequality for $\bar p_n$ follows.

We now prove the first inequality for $\bar p$. As the set $O \setminus D$ is open, for any open set $B \subset \X$, the set $(O \setminus D)  \cap B$ is also open. Since $\bar \gamma_{n_k} \wto \bar \gamma$, by \cite[Theorem 11.1.1]{Dud02}, for any open set $E$, $\bar \gamma (E) \leq \liminf_{k \to \infty} \bar \gamma_{n_k}(E)$. Then, letting $E = \{ (O \setminus D)  \cap B \} \times \A$ for any open set $B$, we have 
$$\textstyle{\bar p((O \setminus D)  \cap B) \leq  \liminf_{k \to \infty} \bar p_{n_k}((O \setminus D)  \cap B) \leq \nu(B),} \quad \forall B \ \text{open}.$$ 
This inequality must also hold for any $B \in \B(\X)$. To see this, first, define $\bar p'(\cdot) = \bar p( (O \setminus D)  \cap \cdot)$, to simplify notation. By~\cite[Theorem~7.1.3]{Dud02}, on a metric space, finite Borel measures are closed regular, which means, in our case, that for any Borel set $B$, $\bar p'(B) = \sup\{ \bar p'(F) \mid F \subset B, F \ \text{closed} \}$ and the same is true for $\nu(B)$. This in turn implies that $\bar p'(B) = \inf\{ \bar p'(F) \mid F \supset B, F \ \text{open} \}$ and the same for $\nu(B)$. Now given $B \in \B(\X)$, for any open set $F \supset B$, we have $\bar p'(F) \leq \nu(F)$ as proved earlier, and therefore $\bar p'(B) \leq \nu(B)$.
\end{proof}

We now proceed to prove (\ref{eq-prf-su1}) and then (\ref{eq-prf-su2}).

\begin{lemma} \label{lem-prf-su1}
The inequality (\ref{eq-prf-su1}) holds.
\end{lemma}

\begin{proof}
For $m \geq 0$, define $c^{m} : \X \times \A \to \R$ by $c^m(x, a) = \min\{ c(x,a), m\}$. 
Since $\liminf_{k \to \infty} \int c \, d \bar \gamma_{n_k} \, \geq \,   \liminf_{k \to \infty} \int c^m \, d \bar \gamma_{n_k}$ and
$\int c^m \, d \bar \gamma \uparrow \int c \, d \bar \gamma$ as $m \to \infty$ by the monotone convergence theorem, 
to prove (\ref{eq-prf-su1}), it suffices to prove
\begin{equation} \label{eq-prf-su3}
 \liminf_{k \to \infty} \int c^m \, d \bar \gamma_{n_k} \geq \int  c^m \, d \bar \gamma.
\end{equation}

To compare the integrals in (\ref{eq-prf-su3}), consider an arbitrary $\epsilon > 0$. There exists a sufficiently large $j$ such that 
for the compact set $\Gamma_{j}$ in Assumption~\ref{cond-pc-3}(SU), 
its complement satisfies that
\begin{equation} \label{eq-prf-su4}
 \bar \gamma_{n}\big( \Gamma^c_{j} \big)   \leq \epsilon, \quad \forall \, n \geq 1, \qquad \text{and} \qquad \bar \gamma\big( \Gamma_{j}^c \big)  \leq \epsilon.
\end{equation}
In the above, the existence of such $j$ and the first inequality in (\ref{eq-prf-su4}) follow from Assumption~\ref{cond-pc-3}(SU) and the fact $\sup_{n} \int c \, d \bar \gamma_n < \infty$. The second inequality in (\ref{eq-prf-su4}) follows from the fact that $\Gamma_{j}^c$ is an open set and hence, as the weak limit of $\{\bar \gamma_{n_k}\}$, $\bar \gamma$ satisfies that $\bar \gamma\big( \Gamma_{j}^c\big) \leq \liminf_{k \to \infty}  \bar \gamma_{n_k}\big( \Gamma_{j}^c \big)$ by \cite[Theorem 11.1.1]{Dud02}.

Fix this $j$. We use (\ref{eq-prf-su4}) to bound the integrals $\int_{\Gamma_{j}^c} c^m \, d \bar \gamma_{n_k}$ and $\int_{\Gamma_{j}^c}  c^m \, d \bar \gamma$ by $m \, \epsilon$.

We now use Lusin's theorem to handle the integrals of $c^m$ on $\Gamma_{j}$.
Define compact sets $K : = \proj_\X (\Gamma_{j}) \subset \X$ and $F:=\proj_\A (\Gamma_{j}) \subset \A$.
Let $O \supset K$ be the open set, $D$ the closed set, and $\nu$ the finite measure in Assumption~\ref{cond-pc-3}(M) for the given $K$.
By Assumption~\ref{cond-pc-3}(M), $c$ is lower semicontinuous on the closed set $D \times \A$; therefore, so is $c^m$ on $D \times \A$.
Since the set $F$ is finite, applying Lusin's Theorem~\cite[Theorem 7.5.2]{Dud02}, for any $\delta > 0$, we can choose a closed set $B \subset \X$ with $\nu(\X \setminus B) \leq \delta$ such that $c^m$ is continuous on $B \times F$.%footnote starts
\footnote{Details: We apply Lusin's Theorem for each $a \in F$ to obtain a closed set $B_a \subset \X$ such that $\nu(\X \setminus B_a) \leq \delta/|F|$ and $c^m(\cdot, a)$ is continuous on $B_a$. We then take $B = \cap_{a \in F} B_a$.\label{footnote-lysin}}
%footnote ends
Then $c^m$ is lower semicontinuous on the closed set $(D \cup B) \times F$. 
By the Tietze-Urysohn extension theorem~\cite[Theorem 2.6.4]{Dud02} and an approximation property for lower semicontinuous functions \cite[Lemma 7.14]{bs}, the restriction of $c^m$ to $(D \cup B) \times F$ can be extended to a nonnegative lower semicontinuous function $\tilde c^m$ on $\X \times \A$ with the extension also bounded above by $m$.%footnote starts
\footnote{Details: Denote the restriction of $c^m$ to $(D \cup B) \times F$ by $f$. Since it is lower semicontinuous and nonnegative, by \cite[Lemma 7.14]{bs}, there exists a sequence of nonnegative continuous functions $\{f_n\}$ on $(D \cup B) \times F$ with $f_n \uparrow f$. 
We apply the Tietze-Urysohn extension theorem~\cite[Theorem 2.6.4]{Dud02} to extend each $f_n$ to a continuous function $\tilde f_n$ on $\X \times \A$ that is also nonnegative and bounded above by $m$. We then let $\tilde c^m = \sup_{n} \tilde f_n$.}
%footnote ends 
For the function $\tilde c^m$, since $\bar \gamma_{n_k} \wto \bar \gamma$, by \cite[Prop.\ E.2]{HL96},
\begin{equation} \label{eq-prf-su5}
\liminf_{k \to \infty} \int \tilde c^m d \bar \gamma_{n_k} \geq \int \tilde c^m d \bar \gamma.
\end{equation}

We now compare the integrals of $c^m$ with those of $\tilde c^m$ and bound their differences:
\begin{align}
  \left| \, \int_{\X \times \A}  \big(c^m - \tilde c^m\big)  \, d \bar \gamma_{n_k}  \,  \right| & \, = \,  \left| \, \int_{\big((D \cup B) \times F\big)^c}  \big(c^m - \tilde c^m\big) \, d \bar \gamma_{n_k}  \,  \right|  \notag \\
  & \, \leq \, \int_{\big(K \setminus (D \cup B) \big) \times F}  \big| c^m - \tilde c^m \big| \, d \bar \gamma_{n_k}    + \int_{(K \times F)^c}  \big| c^m - \tilde c^m \big| \, d \bar \gamma_{n_k} \notag \\
  & \, \leq \, m \int_{\big(O \setminus (D \cup B) \big) \times \A} \, d \bar \gamma_{n_k} + m \,\epsilon \label{eq-prf-su6a}\\
  & \, \leq \,  m \, \nu(B^c) + m \, \epsilon \label{eq-prf-su6b} \\
  &  \, \leq \, m \, (\delta + \epsilon), \label{eq-prf-su6c}
\end{align}
where we used the facts $K \subset O$, $(K \times F)^c \subset \Gamma_{j}^c$ and the inequality (\ref{eq-prf-su4}) to derive (\ref{eq-prf-su6a}); and we used Lemma~\ref{lem-su-mp} and the fact $\nu(\X \setminus B) \leq \delta$ to derive (\ref{eq-prf-su6b}) and (\ref{eq-prf-su6c}), respectively. By the same arguments, for $\bar \gamma$, we also have
\begin{equation}\label{eq-prf-su6d}
 \left| \, \int_{\X \times \A}  \big(c^m - \tilde c^m\big)  \, d \bar \gamma  \,  \right|  \, \leq \,  m \, (\delta + \epsilon).
\end{equation}
By combining (\ref{eq-prf-su6c})-(\ref{eq-prf-su6d}) with (\ref{eq-prf-su5}), we have
$$ \liminf_{k \to \infty} \int c^m  d \bar \gamma_{n_k} \geq \int c^m d \bar \gamma - 2m \, ( \delta + \epsilon).$$
Since $\delta$ and $\epsilon$ are both arbitrary, we obtain
$\liminf_{k \to \infty} \int c^m  d \bar \gamma_{n_k} \geq \int c^m  d \bar \gamma$, which is (\ref{eq-prf-su3}) and implies the desired inequality (\ref{eq-prf-su1}) as discussed earlier.
\end{proof}

\begin{lemma} \label{lem-prf-su2}
The equality (\ref{eq-prf-su2}) holds.
\end{lemma}

\begin{proof}
Recall that $\bar p_{n}$ and $\bar p$ are the marginals of $\bar \gamma_{n}$ and $\bar \gamma$, respectively, on $\X$.
For any $v \in \C_b(\X)$, since $\bar \gamma_{n_k} \wto \bar \gamma$, the right-hand side of (\ref{eq-prf-su2}) satisfies
$$ \int v \, d \bar p = \lim_{k \to \infty} \int v \, d \bar p_{n_k}.$$
The same proof given in \cite[p.\ 119]{HL96} (which is based on a martingale argument) establishes that
\begin{equation} \label{eq-prf-su6e}
  \lim_{n \to \infty} \left\{ \int_{\X \times \A} \int_{\X} v(y) \, q(dy \mid x, a) \, \bar \gamma_{n} (d(x,a)) -  \int_{\X} v(x) \, \bar p_{n}(dx)  \right\} \, =  \, 0.
\end{equation}  
Therefore, to prove (\ref{eq-prf-su2}), it suffices to show that for any $v \in \C_b(\X)$,
\begin{equation} \label{eq-prf-su6}
\lim_{k \to \infty} \int_{\X \times \A} \int_{\X} v(y) \, q(dy \mid x, a) \, \bar \gamma_{n_k} (d(x,a)) = \int_{\X \times \A} \int_{\X} v(y) \, q(dy \mid x, a) \, \bar \gamma (d(x,a)).
\end{equation}

Let $\epsilon > 0$. Let the sets $\Gamma_{j}$, $K \subset O$, $F$ and $D$, and the finite measure $\nu$ be as in the proof of Lemma~\ref{lem-prf-su1} (in particular, recall that $j$ is sufficiently large so that (\ref{eq-prf-su4}) holds). 
Recall that by Assumption~\ref{cond-pc-3}(M), $q(dy \mid x, a)$ is continuous on $D \times \A$.
Since the space $\P(\X)$ is separable and metrizable \cite[Prop.\ 7.20]{bs} and the set $F$ is finite, applying Lusin's Theorem~\cite[Theorem 7.5.2]{Dud02} (see Footnote~\ref{footnote-lysin}), for any $\delta > 0$, we can choose a closed set $B \subset \X$ with $\nu(\X \setminus B) \leq \delta$ such that $q(dy \mid x, a)$ is continuous on $B \times F$.
Then $q(dy \mid x, a)$ is continuous on the closed set $(D \cup B) \times F$, so by \cite[Prop.~7.30]{bs}, $\phi(x,a) : = \int_\X v(y) \, q(dy \mid x, a)$ is a bounded continuous function on the closed set $(D \cup B) \times F$, and by the Tietze-Urysohn extension theorem~\cite[Theorem~2.6.4]{Dud02}, this restriction of $\phi$ to $(D \cup B) \times F$ can be extended to a continuous function $\tilde \phi$ on $\X \times \A$ with $\|\tilde \phi\|_\infty \leq \| \phi\|_\infty \leq \| v\|_\infty$.

Since $\tilde \phi$ is bounded and continuous and $\bar \gamma_{n_k} \wto \bar \gamma$, we have
\begin{equation} \label{eq-prf-su7}
\lim_{k \to \infty} \int \tilde \phi  \, d \bar \gamma_{n_k} = \int \tilde \phi  \, d \bar \gamma .
\end{equation}
We now compare the integrals of $\phi$ with those of $\tilde \phi$ and bound their differences, similarly to the derivation of (\ref{eq-prf-su6c})-(\ref{eq-prf-su6d}):
\begin{align}
 \left| \int_{\X \times \A} \big( \phi - \tilde \phi \big) \, d \bar \gamma_{n_k}  \right| & \, = \,  \left| \int_{ \big( (D \cup B) \times F \big)^c} \big( \phi - \tilde \phi \big) \, d \bar \gamma_{n_k} \right|  \notag\\
 & \, \leq \, \int_{\big(K \setminus (D \cup B) \big) \times F}  \big| \phi - \tilde \phi \big| \, d \bar \gamma_{n_k}    + \int_{(K \times F)^c}  \big| \phi - \tilde \phi \big| \, d \bar \gamma_{n_k}  \notag\\
 & \, \leq \, 2 \| v\|_\infty \cdot \int_{ \big(O \setminus (D \cup B)\big) \times \A}  d \bar \gamma_{n_k}  + 2 \| v\|_\infty \cdot \epsilon  \notag\\
  & \, \leq \, 2 \| v\|_\infty \cdot \nu(B^c) + 2 \| v\|_\infty \cdot \epsilon  \notag \\
  & \, \leq \, 2 \| v\|_\infty \cdot (\delta + \epsilon), \label{eq-prf-su7a}
\end{align}
where, as before, we used the following sets of relations to derive the last three inequalities, respectively: $K \subset O$ and $(K \times F)^c \subset \Gamma_{j}^c$ together with (\ref{eq-prf-su4}); Lemma~\ref{lem-su-mp}; and the fact $\nu(\X \setminus B) \leq \delta$ by the choice of $B$. By the same arguments, for the integrals w.r.t.\ $\bar \gamma$, we also have
\begin{equation} \label{eq-prf-su7b}
  \left| \int_{\X \times \A} \big( \phi - \tilde \phi \big) \, d \bar \gamma \, \right| \, \leq \, 2 \| v\|_\infty \cdot (\delta + \epsilon).
\end{equation}  
Combining the three relations (\ref{eq-prf-su7})-(\ref{eq-prf-su7b}), we have
$$ \limsup_{k \to \infty} \left| \int \phi  \, d \bar\gamma_{n_k}  - \int  \phi \, d \bar \gamma  \, \right|  \, \leq \, 4 \| v \|_\infty \cdot (\delta + \epsilon ). $$
Since $\delta$ and $\epsilon$ are arbitrary,  we obtain the desired inequality (\ref{eq-prf-su6}), which implies (\ref{eq-prf-su2}), as discussed earlier.
\end{proof}

\begin{proof}[Proof of Prop.~\ref{prp-su-mp}]
The proposition follows from Lemmas~\ref{lem-prf-su1}-\ref{lem-prf-su2} and the discussion given immediately before Lemma~\ref{lem-su-mp}.
\end{proof}

\begin{proof}[Proof of Theorem~\ref{thm-su-mp}]
Apply Prop.~\ref{prp-su-mp} to construct a sequence of stationary pairs $(\bar \mu_n, \bar p_n) \in \Delta_s$ with finite average costs $J(\bar \mu_n, \bar p_n) \downarrow \rho^*$. 
For each pair $(\bar \mu_n, \bar p_n)$, since $\bar \gamma_n(d(x,a)) : = \bar \mu_n(da \,|\, x) \, \bar p_n(dx)$ is an invariant probability measure of the Markov chain induced by $\bar \mu_n$ on the state-action space, $\int c \, d \bar \gamma_n = J(\bar \mu_n, \bar p_n)$. So $\int c \, d \bar \gamma_n$ is bounded by some constant for all $n$. In view of Assumption~\ref{cond-pc-3}(SU), this implies, as in the preceding proofs, that $\{\bar \gamma_n\}$ is tight and there is a subsequence $\bar \gamma_{n_k} \wto \bar \gamma \in \P(\X \times \A)$ with $\bar \gamma(\Gamma) =1$. The rest of the proof now parallels that of Prop.~\ref{prp-su-mp}.
Decompose $\bar \gamma$ into the marginal $\bar p$ on $\X$ and a Borel measurable stochastic kernel $\bar \mu(da \,|\, x)$ on $\A$ given $X$ that obeys the control constraint. To prove the theorem, we need to show that (\ref{eq-prf-su2}) and (\ref{eq-prf-su1}) hold for $\bar \gamma$ and $\{\bar \gamma_{n_k}\}$ in this case.

For all $n \geq 1$, we have $\bar p_n(E)  = \int q(E \,|\, x, a) \, \bar \gamma_n(d(x,a))$ for all $E \in \B(\X)$, by the invariance property of $\bar p_n$. It follows from this relation and Assumption~\ref{cond-pc-3}(M) that the conclusion of Lemma~\ref{lem-su-mp} holds for $\bar p_n$ here, and then the second half of the proof of that lemma shows that its conclusion also holds for $\bar p$ in this case. 
We then use Lemma~\ref{lem-prf-su2} to prove that for any $v \in \C_b(\X)$, (\ref{eq-prf-su2}) holds. 
Since $(\bar \gamma_n, \bar p_n) \in \Delta_s$, instead of (\ref{eq-prf-su6e}), the equality holds for every $n$:
$$ \textstyle{\int_{\X \times \A} \int_{\X} v(y) \, q(dy \mid x, a) \, \bar \gamma_{n} (d(x,a)) =  \int_{\X} v(x) \, \bar p_{n}(dx)},$$
so proving (\ref{eq-prf-su2}) also amounts to showing that (\ref{eq-prf-su6}) holds, and the arguments are the same as those given in the proof of Lemma~\ref{lem-prf-su2}. This establishes that $\bar p$ is an invariant probability measure associated with $\bar \mu$, so $(\bar \mu, \bar p) \in \Delta_s$. 
Finally, the proof for (\ref{eq-prf-su1}) in this case is exactly the same as the proof of Lemma~\ref{lem-prf-su1}, and this establishes that $\int c \, d\bar \gamma = \lim_{k \to \infty} \int c \, d \bar \gamma_{n_k} = \rho^*$.
Hence $(\bar \mu, \bar p)$ is a stationary minimum pair.
\end{proof}

\subsection{Proof of Prop.~\ref{prp-malpha}}

The proof is similar to that of Prop.~\ref{prp-su-mp} except that it involves a different sequence of probability measures, from which a stationary pair will be constructed to have average cost no greater than $\liminf_{\alpha \uparrow 1} \, (1 - \alpha) \, m_\alpha$.

We start similarly to the proof in \cite[Sec.\ 6]{VAm99}.
Let $\alpha \in (0,1)$. For each policy $\pi$ and initial distribution $\zeta$, define a probability measure $\gamma^{(\alpha)}_{\pi, \zeta}$ on $\B(\X \times \A)$ by
\begin{equation} \label{eq-prf-ma-1}
 \gamma^{(\alpha)}_{\pi, \zeta}(B) : = (1 - \alpha) \, \textstyle{ \sum_{n=0}^\infty } \, \alpha^n \, \Pr^\pi_\zeta \big\{  (x_n, a_n) \in B \big\} \,,\qquad \forall \, B \in \B(\X \times \A).
\end{equation} 
Note that $\gamma^{(\alpha)}_{\pi, \zeta}(\Gamma) = 1$ and we can express the $\alpha$-discounted expected cost $v_\alpha(\pi, \zeta)$ as 
$v_\alpha(\pi, \zeta) = (1 - \alpha)^{-1} \int c \, d \gamma^{(\alpha)}_{\pi, \zeta}$.
Note also a relation between $\gamma^{(\alpha)}_{\pi, \zeta}$, the marginal of $\gamma^{(\alpha)}_{\pi, \zeta}$ on $\X$, and the state transition stochastic kernel $q(y \,|\, x,a)$:
\begin{equation} \label{eq-prf-ma2}
 \gamma^{(\alpha)}_{\pi, \zeta}\big(B \times \A\big) = (1 - \alpha) \, \zeta(B) + \alpha \int_{\X \times \A} q(B \mid x, a) \,  \gamma^{(\alpha)}_{\pi, \zeta}\big(d(x,a)\big) \,,\quad \forall \, B \in \B(\X).
\end{equation}  
As in \cite{VAm99}, this relation will be important later in our proof to show that a pair of stationary policy and initial distribution constructed in the proof is a stationary pair.

As discussed before Prop.~\ref{prp-malpha}, by a Tauberian theorem,
$$ \underline{\rho}: = \liminf_{\alpha \uparrow 1} \, (1 - \alpha) \, m_\alpha \leq \limsup_{\alpha \uparrow 1} \, (1 - \alpha) \, m_\alpha  \leq \rho^*.$$
Thus there exists a sequence $\alpha_n \uparrow 1$ and a corresponding sequence $\{(\pi_n, \zeta_n)\}$ of policy and initial distribution pairs such that $(1 - \alpha_n) \, v_{\alpha_n}(\pi_n, \zeta_n) \to \underline{\rho}$. 
In other words, for the corresponding sequence of probability measures $\bar \gamma_n: = \gamma^{(\alpha_n)}_{\pi_n, \zeta_n}$, we have 
\begin{equation} \label{eq-prf-ma3}
 \textstyle{ \int c \, d \bar \gamma_n }  \to \underline{\rho} \leq \rho^* < \infty.
\end{equation}
Then, similarly to the proof for Prop.~\ref{prp-su-mp}, we can extract a weakly convergent subsequence $\{\bar \gamma_{n_k}\}$ and decompose its limit $\bar \gamma$ into the marginal $\bar p$ on $\X$ and a stochastic kernel $\bar \mu(da \,|\, x)$ that corresponds to a stationary policy. 

Let $\bar p_n$ denote the marginal of $\bar \gamma_n$ on $\X$. We shall need the following majorization properties for $\bar p_n$ and $\bar p$, which are similar to those in Lemma~\ref{lem-su-mp}:

\begin{lemma} \label{lem-su-disc-mp}
Let the open set $O$, the closed set $D$, and the finite measure $\nu$ on $\B(\X)$ be as in Assumption~\ref{cond-pc-3}(M) for some $K \in \{\proj_\X(\Gamma_j)\}$. Then for all $B \in \B(\X)$,
$$ \bar p \big(  (O \setminus D) \cap B  \big) \leq \nu(B), \qquad  \bar p_n \big( (O \setminus D) \cap B \big) \leq \nu(B) + (1 - \alpha_n), \quad  \forall \, n \geq 1.$$
\end{lemma}

\begin{proof}
The inequality for $\bar p_n$ follows from (\ref{eq-prf-ma2}) and Assumption~\ref{cond-pc-3}(M). Since $\alpha_n \to 1$ and $\bar \gamma_{n_k} \wto \bar \gamma$, this inequality implies $\bar p \big((O \setminus D) \cap B  \big) \leq \nu(B)$ for all open sets $B$. Then the same relation must hold for all Borel sets $B$, as we showed in the second half of the proof of Lemma~\ref{lem-su-mp}.
\end{proof} 

We can now proceed as in the proof of Prop.~\ref{prp-su-mp}. We need to show that (\ref{eq-prf-su2}) and (\ref{eq-prf-su1}) hold for $\bar \gamma$ and $\{\bar \gamma_{n_k}\}$ in this case. To prove (\ref{eq-prf-su1}), we argue as in the proof of Lemma~\ref{lem-prf-su1}, except that we use Lemma~\ref{lem-su-disc-mp} in place of Lemma~\ref{lem-su-mp}. The former lemma differs from the latter in the extra term $(1 - \alpha_n)$ in the majorization inequality for $\bar p_n$. 
However, since $1- \alpha_n \to 0$, the proof of Lemma~\ref{lem-prf-su1} can be obviously modified to incorporate this diminishing term for the case considered here. The result is the inequality (\ref{eq-prf-su1}), that is, 
\begin{equation} \label{eq-prf-ma4a}
\int c \, d\bar \gamma \leq \lim_{k \to \infty} \int c \, d \bar \gamma_{n_k} = \underline{\rho}.
\end{equation}

To prove (\ref{eq-prf-su2}), as in \cite[Sec.\ 6]{VAm99}, we start with the observation that 
for any $v \in \C_b(\X)$ and $n \geq 1$, by (\ref{eq-prf-ma2}), 
$$  \int_\X v(x) \, \bar p_n(dx) = (1 - \alpha_n) \, \int_\X v(x) \,  \zeta_n(dx)  + \alpha_n \int_{\X \times \A} \int_{\X} v(y) \, q(dy \mid x, a) \, \bar \gamma_n \big(d(x,a)\big) \,,$$
and since $\alpha_n \to 1$, this implies that
\begin{equation} \label{eq-prf-ma4}
 \lim_{n \to \infty} \left\{ \int_{\X \times \A} \int_{\X} v(y) \, q(dy \mid x, a) \, d\bar \gamma_n\big(d(x,a)\big) - \int_\X v(x) \,  \bar p_n(dx)  \right\} = 0.
\end{equation}  
Using (\ref{eq-prf-ma4}) (which has an identical expression as~(\ref{eq-prf-su6e})), we can now proceed as in the proof of Lemma~\ref{lem-prf-su2}, except that we apply Lemma~\ref{lem-su-disc-mp} instead of Lemma~\ref{lem-su-mp} and take care of the slight difference between the two lemmas, as explained above. This gives us $(\bar \mu, \bar p) \in \Delta_s$, 
and therefore $J(\bar \mu, \bar p) = \int c \, d\bar \gamma \leq \underline{\rho} \leq \rho^*$ by (\ref{eq-prf-ma4a}) and (\ref{eq-prf-ma3}). Since $J(\bar \mu, \bar p) \geq \rho^*$, it follows that $\underline{\rho} = \rho^*$. The proof of Prop.~\ref{prp-malpha} is now complete.

\subsection{Proof of Theorem~\ref{thm-su-mp2}}

\subsubsection{Part (a)}
The inequality (\ref{eq-thm2-a1}) in Theorem~\ref{thm-su-mp2}(a),
$$  \iJ(\pi, \zeta) \geq \rho^*, \qquad \forall \, \pi \in \Pi, \ \zeta \in \P(\X),$$
has, in fact, already been established in the proof of Prop.~\ref{prp-su-mp}, where a stationary pair $(\bar \mu, \bar p)$ was constructed to have $J(\bar \mu, \bar p) \leq \iJ(\pi, \zeta)$, as we recall. 

We now prove the inequality (\ref{eq-thm2-a2}) in Theorem~\ref{thm-su-mp2}(a) concerning pathwise average costs:
for all $(\pi, \zeta) \in \Pi \times \P(\X)$,
$$ \liminf_{n \to \infty} \, n^{-1} \textstyle{ \sum_{k=0}^{n-1} c(x_k, a_k)} \geq \rho^*, \qquad \text{$\Pr^{\pi}_\zeta$-almost surely}.$$
The idea of the proof is the same as that of \cite[Theorem 3.4]{VAm99} and similar to that of Prop.~\ref{prp-su-mp}: it is to construct a stationary pair $(\bar \mu, \bar p)$ with 
$$J(\bar \mu, \bar p) \leq  \liminf_{n \to \infty} \, n^{-1} \textstyle{ \sum_{k=0}^{n-1} c(x_k, a_k)}$$
for each sample path, except for the sample paths from a set of $\Pr^{\pi}_\zeta$-measure $0$. 

To prepare for the proof, we first define some notations and give two lemmas that we will need.
Consider the process $\{(x_n,a_n)\}$ induced by an arbitrary pair $(\pi, \zeta) \in \Pi \times \P(\X)$. 
Let $\Omega$ denote the sample space and $\omega$ a point in $\Omega$.
For $n \geq 1$, define occupancy measures $\bar \gamma_n^\omega$, which are $\P(\X \times \A)$-valued random variables, by 
$$ \textstyle{\bar \gamma_n^\omega(B) : = \frac{1}{n} \sum_{k=1}^n \ind\big[(x_k, a_k) \in B\big]}, \qquad B \in \B(\X \times \A),$$
where $\ind[ \cdot]$ is the indicator function.
Let $\bar p_n^\omega$ denote the marginal of $\bar \gamma_n^\omega$ on $\X$.

We will need the following lemma. It is similar to but differs from Lemma~\ref{lem-su-mp} mostly in that the desired majorization property holds for \emph{each} Borel set almost surely.

\begin{lemma} \label{lem-su-pathwise-mp}
Let the open set $O$, the closed set $D$, and the finite measure $\nu$ on $\B(\X)$ be as in Assumption~\ref{cond-pc-3}(M) for some $K \in \{\proj_\X(\Gamma_j)\}$. Then for each $E \in \B(\X)$, 
\begin{equation} \label{eq-lem-su-pathwise-mp}
 \limsup_{n \to \infty} \, \bar p_n^\omega \big( (O \setminus D) \cap E \big) \leq \nu(E), \qquad \Pr^\pi_\zeta\text{-almost surely}.
\end{equation} 
\end{lemma} 

\begin{proof}
For a Borel set $E \subset \X$ and $n \geq 1$, 
let $Y_n := \ind\big[x_n  \in (O \setminus D) \cap E \,\big]$ and
$$ \textstyle{Z_n := \sum_{k=1}^n \E^{\pi}_\zeta \big[ Y_k \mid \F_{k-1} \big], \qquad S_n := \sum_{k=1}^n \big(  Y_k - \E^{\pi}_\zeta \big[ Y_k \mid \F_{k-1} \big] \big),}$$
where $\F_k$ denotes the $\sigma$-algebra generated by $(x_0, a_0, \ldots, x_k, a_k)$. 
Then 
$$ \bar p_n^\omega \big( (O \setminus D) \cap E \big) = n^{-1} Z_n + n^{-1} S_n.$$
By the Markov property and Assumption~\ref{cond-pc-3}(M), almost surely,
$$ \E^{\pi}_\zeta \big[ Y_k \mid \F_{k-1} \big] = q\big( (O \setminus D) \cap E \mid x_{k-1}, a_{k-1} \big) \leq \nu(E), \quad \forall \, k \geq 1,$$
and therefore, $n^{-1} Z_n \leq \nu(E)$ almost surely.
Since the $Y_k$'s lie in $[0,1]$, $S_n$ is the partial sum of a bounded Martingale difference sequence, so $n^{-1} S_n \to 0$ almost surely by \cite[Theorem 2.18]{HaH80}.  Hence $\limsup_{n \to \infty} \, \bar p_n^\omega \big( (O \setminus D) \cap E \big) \leq \nu(E)$ almost surely.
\end{proof}

We will also need a known relation similar to (\ref{eq-prf-su6e}), but its proof is slightly different (see \cite[p.\ 375-376]{VAm99}; \cite[p.\ 195-196]{HL99}) and is given below for completeness. 
The difference between (\ref{eq-prf-su6e}) and this relation (\ref{eq-lem-su-pathwise-invm}) is that whereas the former holds for all functions in $\C_b(\X)$, the latter holds for \emph{each} function in $\C_b(\X)$ almost surely.

\begin{lemma} \label{lem-su-pathwise-invm}
For each $v \in \C_b(\X)$, $\Pr^\pi_\zeta$-almost surely,
\begin{equation} \label{eq-lem-su-pathwise-invm}
  \lim_{n \to \infty} \left\{ \int_{\X \times \A} \int_{\X} v(y) \, q(dy \mid x, a) \, \bar \gamma^\omega_{n} (d(x,a)) -  \int_{\X} v(x) \, \bar p^\omega_{n}(dx)  \right\} \, =  \, 0.
\end{equation}  
\end{lemma}

\begin{proof}
Using the definition of the occupancy measures $\bar \gamma^\omega_n$ and $\bar p_n^\omega$, we can rewrite (\ref{eq-lem-su-pathwise-invm}) as that $W_n \to 0$ almost surely, for
$$ W_n : = n^{-1} S_n + n^{-1} \big\{ \textstyle{ \int_\X v(y) \, q(dy \mid x_{n}, a_{n}) -  \int_\X v(y) \, q(dy \mid x_{0}, a_{0}) } \big\}$$
where
$$ S_n : = \textstyle{  \sum_{k=1}^n  \big\{ \int_\X v(y) \, q(dy \mid x_{k-1}, a_{k-1}) - v(x_k) \big\}}.$$
Then since $v$ is bounded, $S_n$ is the partial sum of a bounded Martingale difference sequence, and hence $n^{-1} S_n \to 0$ almost surely by \cite[Theorem 2.18]{HaH80}. Consequently, $W_n \to 0$ almost surely.
\end{proof}

We are now ready to prove (\ref{eq-thm2-a2}). The main proof arguments involve the use of Lusin's theorem and are essentially the same as those for Prop.~\ref{prp-su-mp}. However, the details are somewhat different, because in this case, most arguments hold only almost surely, so we need to be careful that in the proof, we exclude, in total, only a countable number of $\Pr^\pi_\zeta$-null sets on which the desired arguments do not hold. 

Let us start by specifying those sets of sample paths that will be excluded from consideration. 
The first null set to exclude is 
$$\N_0 : = \big\{ \omega \in \Omega \mid (x_n, a_n) \not\in \Gamma \ \text{for some } n \big\}.$$
Let $\{v_\ell \}$ be a sequence of functions in $\C_b(\X)$ with the property in (\ref{eq-detclass}); that is, for any two Borel probability measures $p, p'$ on $\X$,
\begin{equation} \label{eq-detclass-v}
  p = p'  \qquad \Longleftrightarrow \qquad \textstyle{\int v_\ell \, d p \, = \, \int v_\ell \, dp'}, \quad \forall \, \ell \geq 1.
\end{equation}  
The second null set to exclude is 
$$ \N_1 : = \big\{  \omega \in \Omega \, \big| \,  \text{the equality (\ref{eq-lem-su-pathwise-invm}) in Lemma~\ref{lem-su-pathwise-invm} is violated for some $v \in \{v_\ell\}$}  \big\}.$$

To define the next null set $\N_2$ to exclude, we need more notation and definitions. Recall that for $m \geq 0$, the truncated one-stage cost function $c^m(\cdot) = \min \{c(\cdot), m\}$. Let $\Z_+$ be the set of all positive integers. 
For each $j \in \Z_+$, corresponding to the compact set $\Gamma_j$ in Assumption~\ref{cond-pc-3}(SU), let $(O_j, D_j, \nu_j)$ be the open set, the closed set, and the finite measure, respectively, in Assumption~\ref{cond-pc-3}(M) for $K = \proj_\X(\Gamma_j)$; 
and let $F_j$ be the compact (finite) action set $F_j : = \proj_\A(\Gamma_j)$.
For each $j,m \in \Z_+$, choose closed subsets $B^1_{i,j,m}$ and $B^2_{i,j}$ of $\X$ for $i \in \Z_+$ such that the following hold:
\begin{enumerate}[leftmargin=0.7cm,labelwidth=!]
\item[(i)] $\nu_j \big(\X \setminus B^1_{i,j,m} \big) \leq i^{-1}$ and $\nu_j \big(\X \setminus B^2_{i,j} \big) \leq i^{-1}$;
\item[(ii)] restricted to the set $B^1_{i,j,m} \times F_j$, the function $c^m(\cdot)$ is continuous, and restricted to the set $B^2_{i,j} \times F_j$, the state transition stochastic kernel $q(dy \,|\, x,a)$ is continuous.
\end{enumerate}
This is possible by Lusin's theorem, as in the proofs of Lemmas~\ref{lem-prf-su1}-\ref{lem-prf-su2}. Now define two countable collections
\begin{equation} 
  \W_1 := \big\{ (O_j, D_j, \nu_j, B^1_{i,j,m}) \mid i,j,m \in \Z_+ \big\}, \quad   \W_2 := \big\{ (O_j, D_j, \nu_j, B^2_{i,j}) \mid i,j \in \Z_+ \big\}, \notag
\end{equation}
and let the third null set be
\begin{align*}
  \N_2 : = \Big\{   \omega \in \Omega \ \Big| \  & \text{for some $(O,D,\nu, B) \in \W_1 \cup \W_2\,$, the inequality (\ref{eq-lem-su-pathwise-mp}) } \\ 
  & \text{in Lemma~\ref{lem-su-pathwise-mp} is violated by} \ (O,D,\nu) \ \text{and} \ E = B^c  \ \Big\}.
\end{align*}  
The $\Pr^\pi_\zeta$-null set $\N : = \N_0 \cup \N_1 \cup \N_2$ will be excluded from consideration.
 
Consider now an arbitrary (fixed) $\omega \in \Omega \setminus \N$. 
Let  
$$\rho^\omega:= \liminf_{n \to \infty} \, n^{-1} \textstyle{ \sum_{k=0}^{n-1} c(x_k, a_k)}.$$
If $\rho^\omega = +\infty$, the desired inequality (\ref{eq-thm2-a2}) holds trivially, so let us suppose 
$\rho^\omega < \infty$.
Then there exists a subsequence $\{\bar \gamma_{n_k}^\omega\}$ with
$$ \lim_{k \to \infty} \int c \, d \bar \gamma_{n_k}^\omega = \rho^\omega < \infty,$$
which also implies $\sup_k \int c \, d \bar \gamma_{n_k}^\omega < \infty$. By the definition of $\N$, all $\bar \gamma_{n_k}^\omega$ are concentrated on $\Gamma$. As in the proof of Prop.~\ref{prp-su-mp}, since $c$ is strictly unbounded under Assumption~\ref{cond-pc-3}(SU), it follows that $\{\bar \gamma_{n_k}^\omega\}$, restricted to $\Gamma$, is tight in $\P(\Gamma)$, so we can extract a further subsequence, also denoted by $\{\bar \gamma_{n_k}^\omega\}$ (for notational simplicity), that converges weakly to some $\bar \gamma^\omega \in \P(\X \times \A)$ with $\bar \gamma^\omega(\Gamma) = 1$. 
Before proceeding, note that we are now working with a \emph{fixed} $\omega$; we do \emph{not} require the $n_k$'s and $\bar \gamma^\omega$ to be random variables (i.e., measurable w.r.t.\ $\Pr^\pi_\zeta$).

As before, $\bar \gamma^\omega$ can be decomposed into the marginal $\bar p^\omega$ on $\X$ and a Borel measurable stochastic kernel $\bar \mu^\omega(da \,|\, x)$ on $\A$ given $\X$ that obeys the control constraint:
$$ \bar \gamma^\omega\big(d(x,a)\big) =  \bar \mu^\omega(da \mid x) \, \bar p^\omega(dx) \quad \text{and} \quad \bar \mu^\omega\big(A(x) \mid x\big) = 1, \ \ \ \forall \, x \in \X.$$
We are going to show that $(\bar \mu^\omega, \bar p^\omega)$ is a stationary pair with average cost 
$$J(\bar \mu^\omega, \bar p^\omega) \leq \lim_{k \to \infty} \int c \, d \bar \gamma_{n_k}^\omega.$$
To this end, let us summarize some properties of $\omega$ that we will use:

\begin{lemma} \label{lem-prf-pw0}
For the preceding $\omega \in \Omega \setminus \N$, the following hold:
\begin{enumerate}[leftmargin=0.65cm,labelwidth=!]
\item[\rm (a)] 
For every $(O, D, \nu, B) \in \W_1 \cup \W_2$,
\begin{equation}
 \limsup_{k \to \infty} \, \bar p^\omega_{n_k} \big((O \setminus D)\cap B^c \big) \leq \nu(B^c), \qquad   \bar p^\omega\big( (O \setminus D)\cap B^c \big) \leq \nu(B^c). \notag
\end{equation}
\item[\rm (b)] For every $\ell \in \Z_+$,
\begin{equation}
     \lim_{n \to \infty} \left\{ \int_{\X \times \A} \int_{\X} v_{\ell}(y) \, q(dy \mid x, a) \, \bar \gamma^\omega_{n} (d(x,a)) -  \int_{\X} v_{\ell}(x) \, \bar p^\omega_{n}(dx)  \right\} \, =  \, 0. \notag
\end{equation} 
\end{enumerate}
\end{lemma}

\begin{proof} 
The lemma follows directly from the definition of the null set $\N$, except for the majorization inequality for $\bar p^\omega$ in (a). For this inequality, note that $(O \setminus D)\cap B^c$ is an open set and $\bar p^\omega_{n_k} \wto \bar p^\omega$. Therefore, by \cite[Theorem 11.1.1]{Dud02} and the first relation in (a),
$$\bar p^\omega\big( (O \setminus D)\cap B^c \big) \leq \liminf_{k \to \infty} \bar p^\omega_{n_k} \big((O \setminus D)\cap B^c \big) \leq \nu(B^c).\qedhere$$
\end{proof}

%%\smallskip

The next two lemmas are the sample-path analogues of Lemmas~\ref{lem-prf-su1}-\ref{lem-prf-su2} in the proof of Prop.~\ref{prp-su-mp}, and they will lead to the desired inequality (\ref{eq-thm2-a2}).

\begin{lemma} \label{lem-prf-pw1}
For the preceding $\omega \in \Omega \setminus \N$, 
$ \int  c \, d \bar \gamma^\omega \leq \rho^\omega.$
\end{lemma}

\begin{proof}
The proof is similar to that of Lemmas~\ref{lem-prf-su1}, although we will need a few changes. Besides working with $\bar \gamma_{n_k}^\omega$ and $\bar \gamma^\omega$ instead of the $\bar \gamma_{n_k}$ and $\bar \gamma$ in that proof, we will take the numbers $m, \delta, \epsilon$ in that proof to be from countable sets: let $m \in \Z_+$ and $\delta, \epsilon \in \Q0:=\{ i^{-1} \mid i \in \Z_+ \}$. 
We will also use the majorization property given by Lemma~\ref{lem-prf-pw0}(a) instead of Lemma~\ref{lem-su-mp}.

With these changes, we proceed as in the proof of Lemmas~\ref{lem-prf-su1}. In particular:
\begin{enumerate}[leftmargin=0.65cm,labelwidth=!]
\item We first replace $c$ by $c^m$, for an arbitrarily large $m \in \Z_+$. 
\item We then consider arbitrarily small $\epsilon = \delta = i^{-1} \in \Q0$, 
choose $j$ large enough so that (\ref{eq-prf-su4}) holds for $\Gamma^c_j$ and $\epsilon = i^{-1}$, 
and corresponding to the given $m, j$ and $\delta = i^{-1}$, choose from the set $\W_1$ the element $(O,D,\nu,B) = (O_j, D_j, \nu_j, B_{i,j,m}^1)$. 
\item Using the property of $(O,D,\nu,B)$ (see the definition of $B_{i,j,m}^1$), 
we construct the function $\tilde c^m$ and obtain the inequality (\ref{eq-prf-su5}) as before.
\item Using the majorization property given in Lemma~\ref{lem-prf-pw0}(a), with the same reasoning as before, 
we obtain that the inequality (\ref{eq-prf-su6d}) holds for $\bar \gamma^\omega$, and that the inequality (\ref{eq-prf-su6c}) holds for $\bar \gamma^\omega_{n_k}$ in the limit: $\limsup_{k \to \infty} \big| \int \big(c^m - \tilde c^m\big) \, d \bar \gamma_{n_k}^\omega \big| \leq m (\delta + \epsilon)$.
\end{enumerate}
Combining the results of the last two steps, we obtain
$$  \liminf_{k \to \infty} \int c^m  d \bar \gamma_{n_k}^\omega \geq \int c^m  d \bar \gamma^\omega - 2m \, ( \delta + \epsilon).$$
The lemma then follows, as before, by letting $\delta, \epsilon \to 0$ and then $m \to \infty$.
\end{proof}

\begin{lemma} \label{lem-prf-pw2}
For the preceding $\omega \in \Omega \setminus \N$, $(\bar \mu^\omega, \bar p^\omega) \in \Delta_s$.
\end{lemma}

\begin{proof}
The proof is similar to that of Lemma~\ref{lem-prf-su2}, other than a few changes like in the proof for Lemma~\ref{lem-prf-pw1}. We will let $\delta, \epsilon \in \Q0=\{ i^{-1} \mid i \in \Z_+ \}$ and work with $\bar \gamma_{n_k}^\omega$ and $\bar \gamma^\omega$ instead of the $\bar \gamma_{n_k}$ and $\bar \gamma$.
We will use Lemma~\ref{lem-prf-pw0}(a) instead of Lemma~\ref{lem-su-mp} for the needed majorization property, and we will use Lemma~\ref{lem-prf-pw0}(b) instead of (\ref{eq-prf-su6e}).

With these changes, we proceed as follows. 
To prove $\bar p^\omega$ is the invariant probability measure of the Markov chain $\{x_n\}$ induced by $\bar \mu^\omega$, it suffices to consider the countable family $\{v_\ell\} \subset \C_b(\X)$ with the property (\ref{eq-detclass-v}) and show that (\ref{eq-prf-su2}) holds for all $v \in \{v_\ell\}$: 
\begin{equation} \label{eq-prf-pw1}
\int_{\X \times \A} \int_{\X} v(y) \, q(dy \mid x, a) \, \bar \gamma^\omega (d(x,a)) = \int_{\X} v(x) \, \bar p^\omega(dx).
\end{equation}
In view of Lemma~\ref{lem-prf-pw0}(b) and the fact $\bar p_{n_k}^\omega \wto \bar p^\omega$, to prove (\ref{eq-prf-pw1}), it suffices to show that for all $v \in \{v_\ell\}$,
\begin{equation} \label{eq-prf-pw2}
\lim_{k \to \infty} \int_{\X \times \A} \int_{\X} v(y) \, q(dy \mid x, a) \, \bar \gamma_{n_k}^\omega (d(x,a)) = \int_{\X \times \A} \int_{\X} v(y) \, q(dy \mid x, a) \, \bar \gamma^\omega (d(x,a)).
 \end{equation}

We now prove (\ref{eq-prf-pw2}) for an arbitrary function $v \in \{v_\ell\}$. 
Define a function $\phi(x,a) : = \int_\X v(y) \, q(dy \mid x, a)$ on $\X \times \A$. 
Then proceed as in the proof of Lemma~\ref{lem-prf-su2}:
\begin{enumerate}[leftmargin=0.65cm,labelwidth=!]
\item Consider arbitrarily small $\epsilon = \delta = i^{-1} \in \Q0$.
Choose $j$ large enough so that (\ref{eq-prf-su4}) holds for $\Gamma_j^c$ and $\epsilon = i^{-1}$, 
and corresponding to the given $j$ and $\delta = i^{-1}$, choose from the set $\W_2$ the element $(O,D,\nu,B) = (O_j, D_j, \nu_j, B_{i,j}^2)$. 
\item Using the property of $(O,D,\nu,B)$ (see the definition of $B_{i,j}^2$),
we construct the function $\tilde \phi$ from $\phi$, and obtain the inequality (\ref{eq-prf-su7}) as before.
\item Using the majorization property given in Lemma~\ref{lem-prf-pw0}(a), with the same reasoning as before, 
we obtain that the inequality (\ref{eq-prf-su7b}) holds for $\bar \gamma^\omega$, and that the inequality (\ref{eq-prf-su7a}) holds for $\bar \gamma^\omega_{n_k}$ in the limit: $\limsup_{k \to \infty} \!\big|\! \int \big(\phi - \tilde \phi\big)  \, d \bar \gamma_{n_k}^\omega \big| \leq 2 \| v \|_\infty \cdot (\delta + \epsilon ).$
\end{enumerate}
Combining the results from the last two steps, we obtain, as before,
$$ \limsup_{k \to \infty} \left| \int \phi  \, d \bar\gamma_{n_k}^\omega  - \int  \phi \, d \bar \gamma^\omega  \, \right|  \, \leq \, 4 \| v \|_\infty \cdot (\delta + \epsilon ), $$
and by letting $\delta = \epsilon = i^{-1} \to 0$, we obtain the desired relation (\ref{eq-prf-pw2}). It follows that $(\bar \mu^\omega, \bar p^\omega)$ is a stationary pair.
\end{proof}

The inequality (\ref{eq-thm2-a2}) in Theorem~\ref{thm-su-mp2}(a) then follows:

\begin{proof}[Proof of (\ref{eq-thm2-a2}) in Theorem~\ref{thm-su-mp2}(a)]
For each $\omega \in \Omega \setminus \N$, we have either $\rho^\omega = +\infty$ or by Lemmas~\ref{lem-prf-pw1} and \ref{lem-prf-pw2},
$$\infty > \rho^\omega \geq J(\bar \mu^\omega, \bar p^\omega) \geq \rho^*.$$ 
Since the set $\N$ is $\Pr^\pi_\zeta$-null, we have that (\ref{eq-thm2-a2}) holds $\Pr^\pi_\zeta$-almost surely.
\end{proof}
 
The proof of Theorem~\ref{thm-su-mp2}(a) is now complete. 

\subsubsection{Part (b)} \label{sec-3.3b}
We now proceed to prove Theorem~\ref{thm-su-mp2}(b). The proof involves several concepts about Markov chains: maximal irreducibility measures, Harris recurrence, and $f$-regularity; these concepts are explained in Appendix~\ref{appsec-1}. 

We need some preparation for the proof of the pathwise optimality of the policy $\mu^*$ in a stationary minimum pair, as this part is not as straightforward as the rest.
By assumption the Markov chain $\{x_n\}$ induced by $\mu^*$ is positive Harris recurrent. We want to show that $\mu^*$ is pathwise average-cost optimal, and we will apply the strong law of large numbers (LLN) for positive Harris recurrent Markov chains to prove this. However, if we apply the LLN directly to the Markov chain $\{x_n\}$, unless $\mu^*$ is nonrandomized, what we get is only that for the expected one-stage cost function $c_{\mu^*}$, the limit of the average $n^{-1} \sum_{k=0}^{n-1} c_{\mu^*}(x_k)$ exists almost surely and equals $\rho^*$, for any initial state distribution. So, to deal with $n^{-1} \sum_{k=0}^{n-1} c(x_k,a_k)$, we will instead apply the LLN to the Markov chain $\{(x_n, a_n)\}$ induced by $\mu^*$ on a certain subset of the state-action space. Specifically, we first show that restricted to that set, $\{(x_n, a_n)\}$ is positive Harris recurrent; the proof uses the fact that the action space $\A$ is countable.

\begin{lemma} \label{lem-harris-rec}
Suppose $(\mu^*,p^*) \in \Delta_s$ and $\mu^*$ induces a positive Harris recurrent Markov chain on $\X$.
Then the Markov chain $\{(x_n, a_n)\}$ induced by $\mu^*$ on the space 
$\Gamma_{\mu^*} := \big\{(x, a) \in \Gamma \mid \mu^*(\{a\} \,|\, x) > 0 \big\}$
is also positive Harris recurrent.
\end{lemma}

\begin{proof}
To simplify notation, in this proof we write $\Pr^{\mu^*}$ as $\Pr$ instead, dropping the superscript. The set $\Gamma_{\mu^*}$ is formed by simply excluding, for each state, those actions that $\mu^*$ will never take at that state (this makes sense since $\A$ is countable).
Since $\mu^*(da\,|\,x)$ is a Borel measurable stochastic kernel, $\Gamma_{\mu^*}$ is Borel. Let $\gamma^* \in \P(\X \times \A)$ be given by $\gamma^*(d(x,a)) : = \mu^*(da \,|\, x) \, p^*(dx)$. Clearly, $\gamma^*(\Gamma_{\mu^*}) = 1$ and for all initial states $x$, $\Pr_x\big\{ (x_n, a_n) \not\in \Gamma_{\mu^*} \ \text{for some} \ n\big\} = 0$. So we can treat $\Gamma_{\mu^*}$ as the state space of the Markov chain $\{(x_n,a_n)\}$ induced by $\mu^*$.

Under the positive Harris recurrence assumption on $\{x_n\}$, $\gamma^*$ is the unique invariant probability measure of the Markov chain $\{(x_n, a_n)\}$ on $\Gamma_{\mu^*}$. To prove that this Markov chain is positive Harris recurrent, it suffices to show that for every initial state-action pair $(x, a) \in \Gamma_{\mu^*}$,
\begin{equation} \label{eq-prf-harris1}
\Pr_{(x,a)} \big\{ (x_n, a_n) \in B \ i.o. \, \big\} = 1, \quad \forall \, B \in \B\big(\Gamma_{\mu^*}\big) \ \text{with} \  \gamma^*(B) > 0
\end{equation}
(``i.o.''~stands for ``infinitely often''), 
because this will imply both that the Markov chain is $\gamma^*$-irreducible---then being the invariant probability measure, $\gamma^*$ is necessarily a maximal irreducibility measure~\cite[Cor.~5.2]{Num84}---and that it is Harris recurrent.

To prove (\ref{eq-prf-harris1}), consider an arbitrary set $B \in \B\big(\Gamma_{\mu^*}\big)$ with $\gamma^*(B) > 0$. 
For $x \in \proj_\X(B)$, let $B_{x} : = \{ a \in \A \, | \, (x, a) \in B\}$. The definition of $\Gamma_{\mu^*}$ implies that $\mu^*(B_x \mid x) > 0$ for all $x \in \proj_\X(B)$. 
For $m \geq 1$, define sets
$$ E_m : = \big\{ x \in \X \mid \mu^*(B_x \mid x) \geq 1/m \big\}, \qquad B_m : = \big\{ (x, a) \in B \mid x \in E_m \big\}.$$
These sets increase as $m$ increases, with $E_m \uparrow \text{proj}_{\X}(B)$ and $B_m \uparrow B$. 
Then, since $\gamma^*(B) > 0$, for some sufficiently large $m$, $\gamma^*(B_m) > 0$ and $p^*(E_m) > 0$. 
Consider this $m$. Since $\{x_n\}$ is Harris recurrent and as its invariant probability measure, $p^*$ is a maximal irreducibility measure of $\{x_n\}$~\cite[Cor.~5.2]{Num84}, by \cite[Theorem 9.1.4]{MeT09}, $p^*(E_m) > 0$ implies
$$ \Pr_x \big\{ x_n \in E_m \ i.o. \, \big\} = 1, \qquad \forall \, x \in \X.$$
There are only a countable number of actions that the policy can possibly take at an initial state $x$. Therefore, by the Markov property and by the definition of $\Gamma_{\mu^*}$, for every initial state-action pair $(x,a) \in \Gamma_{\mu^*}$,
$$ \Pr_{(x,a)} \big\{ x_n \in E_m \ i.o. \, \big\} = 1.$$
By the extended Borel-Cantelli lemma (see e.g., \cite[Corollary~2.3, p.~32]{HaH80}), 
this implies
\begin{equation} \label{eq-prf-harris2}
  \textstyle{ \sum_{n=1}^\infty \Pr_{(x,a)} \big\{ x_n \in E_m \mid \F_{n-1} \big\} = + \infty}, \quad \text{$\Pr_{(x,a)}$-almost surely},
\end{equation}
where $\F_{n-1}$ is the $\sigma$-algebra generated by $(x_k, a_k)$, $k \leq n-1$.
Now by the Markov property, for $n \geq 1$, $\Pr_{(x,a)} \big\{ x_n \in E_m \mid \F_{n-1} \big\} = q(E_m \, |\, x_{n-1}, a_{n-1})$ and
\begin{align*} 
\Pr_{(x,a)} \big\{ (x_n,a_n) \in B_m \mid \F_{n-1} \big\} & = \int_{E_m}  \mu^*(B_y \mid y) \, q(dy \mid x_{n-1}, a_{n-1}) \\
& \geq m^{-1} q(E_m \, |\, x_{n-1}, a_{n-1}),
\end{align*} 
where the equality (inequality) follows from the definition of the set $B_m$ ($E_m$). Hence, by (\ref{eq-prf-harris2}), $\Pr_{(x,a)}$-almost surely,
$$   \textstyle{ \sum_{n=1}^\infty \Pr_{(x,a)} \big\{ (x_n,a_n) \in B_m \mid \F_{n-1} \big\}  \geq m^{-1}  \sum_{n=1}^\infty \Pr_{(x,a)} \big\{ x_n \in E_m \mid \F_{n-1} \big\} = + \infty}.$$
This implies, by the extended Borel-Cantelli lemma \cite[Corollary~2.3, p.~32]{HaH80}, 
that
$$ \Pr_{(x,a)} \big\{ (x_n,a_n) \in B_m \ i.o. \, \big\} = 1.$$ 
Since $B_m \subset B$, the desired relation (\ref{eq-prf-harris1}) holds.
\end{proof}
%%\smallskip

\begin{proof}[Proof of Theorem~\ref{thm-su-mp2}(b)]
For a minimum pair $(\mu^*, p^*) \in \Delta_s$, its average cost is 
$$\int_\X \int_\A c(x, a) \, \mu^*(da \,|\, x) \, p^*(dx) = \rho^* < \infty.$$ 
Then, under the positive Harris recurrence assumption in the theorem, by Lemma~\ref{lem-harris-rec} and the LLN for positive Harris Markov chains~\cite[Theorem 17.1.7]{MeT09}, starting from any initial state-action pair $(x,a) \in \Gamma_{\mu^*}$,
$$ \lim_{n \to \infty}  \, n^{-1} \textstyle{ \sum_{k=0}^{n-1} c(x_k, a_k)} = \rho^*, \qquad \text{$\Pr^{\mu^*}_{(x,a)}$-almost surely}.$$
This together with the definition of $\Gamma_{\mu^*}$ implies that for all initial state distribution $\zeta \in \P(\X)$,
$$ \lim_{n \to \infty}  \, n^{-1} \textstyle{ \sum_{k=0}^{n-1} c(x_k, a_k)} = \rho^*, \qquad \text{$\Pr^{\mu^*}_\zeta$-almost surely}.$$
This proves (\ref{eq-thm2-a3}). Then, in view of the inequality (\ref{eq-thm2-a2}) from Theorem~\ref{thm-su-mp2}(a) proved earlier, it follows that $\mu^*$ is pathwise average-cost optimal.

When the Markov chain $\{x_n\}$ induced by $\mu^*$ is $f$-regular with $f = (c_{\mu^*} + 1)$ being finite-valued, by \cite[Theorem 14.3.6(ii)]{MeT09} for $f$-regular Markov chains, we have that 
$$\textstyle{J(\mu^*,x) = \iJ(\mu^*, x) = \int c_{\mu^*} d p^* = \rho^*}, \qquad \forall \, x \in \X.$$ 
Then, by (\ref{eq-thm2-a1}) from Theorem~\ref{thm-su-mp2}(a), $\mu^*$ is strongly average-cost 
optimal.
\end{proof}
%\smallskip

%\appendix 
\begin{appendices}

\section{$\psi$-Irreducible, Harris Recurrent, and Regular Markov Chains} \label{appsec-1}
In this appendix we explain the concepts of $\psi$-irreducible, Harris recurrent, and $f$-regular Markov chains. We refer the reader to the book \cite{MeT09} for further information about these Markov chains.

Consider a Borel space $X$. Let $\phi$ be a nontrivial $\sigma$-finite Borel measure on $X$. A Markov chain $\{x_n\}$ on $X$ is called \emph{$\phi$-irreducible} iff for every $B \in \B(X)$, 
$$\phi(B) > 0 \quad \Longrightarrow \quad  \Pr_x \{ \tau_B < \infty \} > 0,$$ 
where $\tau_B : = \min \{ n \geq 1 \mid x_n \in B \}$ ($\infty$ if the set is empty) is the first return time to the set $B$, 
and $\Pr_x \{ E \}$ denotes the probability of the event $E$ given the initial state $x_0=x$. 
A $\phi$-irreducible Markov chain has a \emph{maximal irreducibility measure} $\psi$---a $\sigma$-finite measure with the property that for every $\phi$ such that the Markov chain is $\phi$-irreducible, $\phi$ is absolutely continuous w.r.t.\ $\psi$. ($\psi$ is not unique and can be chosen to be a finite or probability measure. See~\cite[Prop.~4.2.2]{MeT09} for more details.) A useful fact is that if a $\phi$-irreducible Markov chain admits an invariant probability measure, that probability measure must also be a maximal irreducibility measure for the Markov chain~\cite[Cor.~5.2, p.~74]{Num84}.
Below, the notation $\psi$ will always stand for a maximal irreducibility measure. 
For a $\psi$-irreducible Markov chain, define 
$\B^+(X): = \big\{ B \in \B(X) \mid \psi(B) > 0 \big\}$
(maximal irreducibility measures are equivalent to each other, so $\B^+(X)$ does not depend on the choice of $\psi$).
 
A $\psi$-irreducible Markov chain $\{x_n\}$ on $X$ is called \emph{positive Harris recurrent} iff it satisfies the following~\cite[Chap.\ 9, p.\ 199; Chap.\ 10, p.\ 231]{MeT09}: 
\begin{enumerate}[leftmargin=0.75cm,labelwidth=!]
\item[(i)] It has an invariant probability measure (necessarily unique). 
\item[(ii)] For every $B \in \B^+(X)$, 
$\Pr_x \{ x_n \in B \ \text{i.o.} \} = 1$ for all $x \in B$,
where the abbreviation ``i.o.''~stands for ``infinitely often.''
\end{enumerate}
For $\psi$-irreducible Markov chains, the property (ii) defines a \emph{Harris recurrent} Markov chain, and it is equivalent to that 
$\Pr_x \{ x_n \in B \ \text{i.o.} \} = 1$ for all $x \in X$,
or that 
$\Pr_x \{ \tau_B < \infty  \} = 1$ for all $x \in B$, 
for every $B \in \B^+(X)$, 
despite that the former (latter) requirement seems stronger (weaker) than (ii) (see \cite[Prop.~9.1.1 and Thm.~9.1.4]{MeT09}).

Stronger than positive Harris recurrence is the $f$-regularity property \cite[Chap.~14, p.~339]{MeT09}. Let $f: X \to [1, \infty)$. A $\psi$-irreducible Markov chain $\{x_n\}$ is called \emph{$f$-regular} iff there exists a countable cover of $X$ by $f$-regular sets, where an \emph{$f$-regular} set $C \subset X$ is a Borel measurable set that satisfies 
$$ \textstyle{\sup_{x \in C} \, \E_x \Big[ \sum_{n=0}^{\tau_B -1} f( x_n) \Big]} < \infty, \qquad \forall \, B \in \B^+(X).$$
In the case $f(\cdot) \equiv 1$, the Markov chain is simply called \emph{regular} \cite[Chap.~11]{MeT09}.

Positive Harris recurrent and $f$-regular Markov chains have strong ergodic properties. 
The two ergodic theorems that we use in this paper are~\cite[Theorem 14.3.6(ii) and Theorem 17.1.7]{MeT09} (see the proof of Theorem~\ref{thm-su-mp2}(b) near the end of Section~\ref{sec-3.3b}).

\end{appendices}

\section*{Acknowledgments}
The author thanks Professor Eugene Feinberg, Dr.\ Martha Steenstrup, and the anonymous reviewers for their helpful comments.

\addcontentsline{toc}{section}{References} 
%\markboth{\rm References}{\rm References}
\bibliographystyle{siamplain} %{apa} 
\let\oldbibliography\thebibliography
\renewcommand{\thebibliography}[1]{%
  \oldbibliography{#1}%
  \setlength{\itemsep}{0pt}%
}
{\fontsize{9}{11} \selectfont
\bibliography{acMDP_minpair_bib}}

\end{document}